\documentclass[onefignum,onetabnum]{siamart190516}

\usepackage{lipsum}
\usepackage{amsfonts}
\usepackage{graphicx}
\usepackage{epstopdf}
\usepackage{algorithmic}
\ifpdf
\DeclareGraphicsExtensions{.eps,.pdf,.png,.jpg}
\else
\DeclareGraphicsExtensions{.eps}
\fi

\newsiamremark{remark}{Remark}
\newsiamremark{hypothesis}{Hypothesis}
\crefname{hypothesis}{Hypothesis}{Hypotheses}
\newsiamthm{claim}{Claim}

\usepackage{subcaption}
\usepackage{color}

\newcommand{\eepf}{\hfill $\square$}

\newcommand{\bb}{\mathcal L}

\newcommand{\im}{\mathbf B}
\newcommand{\am}{\mathbf A}

\newcommand{\fp}{\mathbf F}

\newcommand{\mdp}{\mathbf p}  %the p-value of a matrix or diagram
\newcommand{\mdq}{\mathbf q}
\newcommand{\mdr}{\mathbf r}

\newcommand{\properd}{\mathcal P} %proper diagrams
\newcommand{\primed}{\mathcal B}  %regular diagrams

\newcommand{\bmap}{\sigma}  %map a regular diagram to its block matrix
\newcommand{\promap}{\rho} %map a proper diagram to its blcok matrix

\newcommand{\arc}{e}

\newcommand{\blowup}{\delta}

\newcommand{\rev}[1]{#1}%\textcolor{blue}{#1}}

\newtheorem{innercustomthm}{Theorem}
\newenvironment{customthm}[1]
{\renewcommand\theinnercustomthm{#1}\innercustomthm}
{\endinnercustomthm}

\headers{$k$-noncrossing RNA Structures}{V. Moulton and T. Wu}

\title{Posets and Spaces of $k$-noncrossing RNA Structures \thanks{Submitted to the editors DATE.
	}}

\author{Vincent Moulton\thanks{School of Computing Sciences, University of East Anglia, Norwich, U.K.
		(\email{v.moulton@uea.ac.uk})}.
	\and Taoyang Wu\thanks{School of Computing Sciences, University of East Anglia, Norwich, U.K.
		(\email{taoyang.wu@uea.ac.uk})}
}

\usepackage{amsopn}

\ifpdf
\hypersetup{
  pdftitle={An Example Article},
  pdfauthor={D. Doe, P. T. Frank, and J. E. Smith}
}
\fi

\begin{document}

\maketitle

% REQUIRED
\begin{abstract}
RNA molecules are single-stranded analogues of DNA that
can fold into various structures which influence their 
biological function within the cell. 
RNA structures can be modelled combinatorially in
terms of a certain type of graph called an RNA diagram.
In this paper we introduce a new poset of RNA diagrams
$\primed^r_{f,k}$, $r\ge 0$, $k \ge 1$ and $f \ge 3$,
which we call the Penner-Waterman poset,
and, using results from the theory of multitriangulations, 
we show that this is a pure poset of rank $k(2f-2k+1)+r-f-1$,
whose geometric realization
is the join of a simplicial sphere of dimension $k(f-2k)-1$ and 
an $\left((f+1)(k-1)-1\right)$-simplex in case $r=0$.
As a corollary for the special case $k=1$, we  obtain a 
result due to Penner and Waterman
concerning the topology of the space of RNA secondary structures. 
%These results could eventually lead to new ways to explore spaces of RNA structures so as to predict RNA pseudoknots. 
\rev{These results could eventually lead to new ways to study 
landscapes of RNA $k$-noncrossing structures.}
\end{abstract}

% REQUIRED

% REQUIRED

\section{Introduction}
\label{sec:introduction}

Ribonucleic acid (RNA) is a single-stranded polymeric molecule that 
is essential in various biochemical processes within the cell.  
The primary structure of an RNA molecule 
is a linear sequence of four nucleotides, also known as {\em bases} and usually denoted by 
A (adenine), C (cytosine), G (guanine) and U (uracil).  In nature, RNA molecules fold into structures which are
intimately related to their biological function. These structures arise from the
linear sequence folding back onto itself which is possible 
since the non-adjacent bases A/U and G/C can form bonds or {\em base-pairs}.
Motivated by this phenomenon, there has been a great deal of 
interest in the prediction of RNA structures \cite{gorodkin2014rna}, and also in understanding 
combinatorial properties of these structures over the past three decades~\cite{chen2009random,jin2008combinatorics,PW93,reidys2010combinatorial,reidys2002combinatorial}. 

RNA structures are commonly represented by  {\em binary diagrams},  simple graphs drawn in the 
plane in which each base is represented by a vertex and the underlying linear molecule  
is represented by a collection of edges along a horizontal line.  In addition, the base-pairs 
are represented by semi-circles or {\em arcs} in the upper halfplane such that arcs
do not connect the last and first bases and no two arcs are incident with the same vertex
(see e.g.  Fig.~\ref{f:RNA:knot}(i)). In the special case
where the diagram representing an RNA structure has no intersecting arcs, 
the RNA structure is also known as a {\em secondary structure}.  RNA 
secondary structures are important as they often form a backbone structure for an
RNA molecule, and their nested structure facilitates 
their prediction using free-energy models \cite[pp. 1-31]{gorodkin2014rna}.

\begin{figure}[ht]
	\begin{center}
		\includegraphics[width=0.9\textwidth]{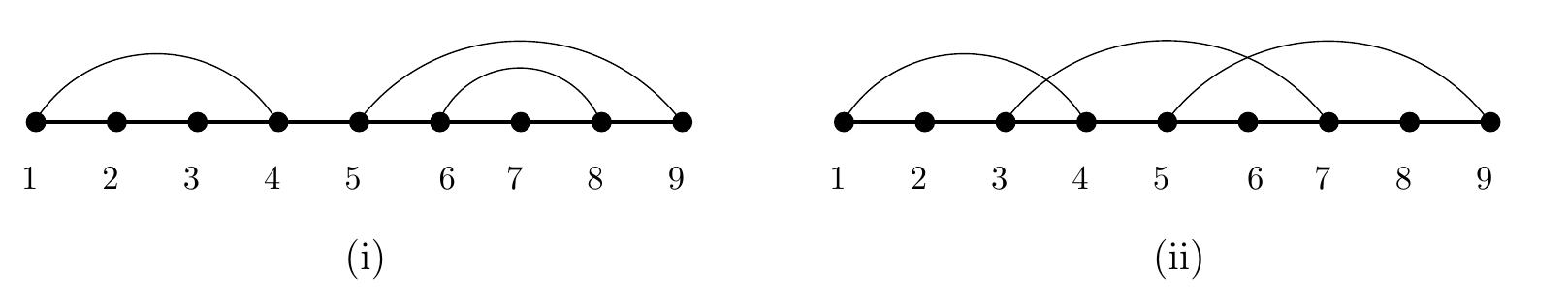}
	\end{center}
	\caption{ (i) An RNA secondary structure with 9 bases (represented by black vertices) 
		and 3 base-pairs (represented by arcs). In this diagram, the arc $(1,4)$ covers free sites 2 and 3, and the arc (6,8) is below the arc (5,9). 
		(ii) An RNA pseudoknot that is a $2$-noncrossing structure with 3 base-pairs.}
	\label{f:RNA:knot}
\end{figure}

In \cite{PW93} Penner and Waterman introduce and study a 
certain {\em space} of RNA secondary structures. 
In combinatorial language, this space can be considered as the geometric realization of 
a certain finite poset $\primed^r_{f}$ of RNA secondary structures for $r\ge 0$ and $f \ge 3$
whose definition we now recall. We call a base in a diagram a {\em free site} if it is not contained in any arc,
and an arc $e$ {\em tautological} if either $e$ covers precisely one free site or there 
exists an arc below $e$ that covers the same set of free sites as $e$.
The poset $\primed^r_{f}$ then consists of all RNA secondary structures whose diagrams have $f$ free sites  
and no more than $r$ tautological arcs.
For example, the RNA secondary structure in  Fig.~\ref{f:RNA:knot}(i) is in $\primed^2_3$. 
The poset relation on $\primed^r_{f}$  is induced by 
\emph{arc suppression}, that is, removal of an arc 
and the two sites which it contains from a diagram (for the formal definition see Section~\ref{subsec:diagram}).  
In Fig.~\ref{f:space}(i) we present the Hasse diagram of the poset $\primed^0_{4}$.

%\bigskip

\begin{figure}[ht]
	\begin{center}
		\includegraphics[width=1\textwidth]{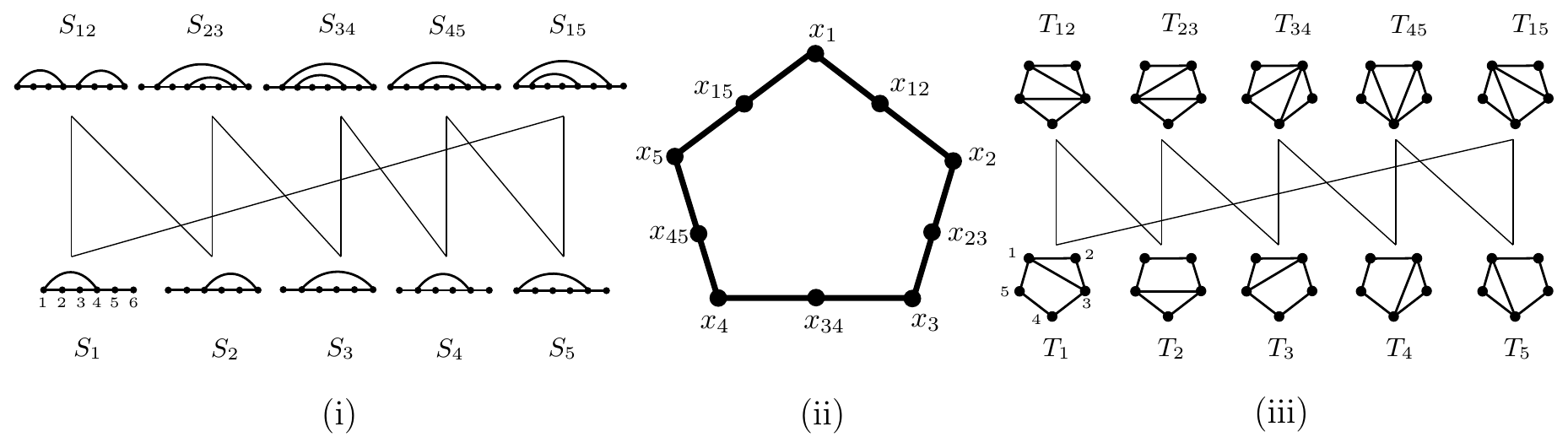}
	\end{center}
	\caption{ (i) The Hasse diagram of the poset $\primed^0_{4}$. (ii) The geometric realization of the 
		poset  $\primed^0_{4}$.  Here the vertices $x_i$ and $x_{ij}$ correspond to 
		RNA secondary structures $S_i$, $S_{ij}$ with the same indices. 
		(iii) The Hasse diagram of the poset $\mathcal T_{5}$, whose elements 
		are the triangulations of a pentagon. The poset isomorphism between 
		$\primed^0_{4}$ and $\mathcal T_{5}$ mentioned in 
		the text maps each secondary structure  in $\primed^0_{4}$ to the 
		triangulation in $\mathcal T_{5}$ with the same indices.}
	\label{f:space}
\end{figure}

%In \cite{PW93} 
Penner and Waterman show 
that the poset $\primed^0_{f}$  is isomorphic to what they call 
the {\em arc poset} $\mathcal T_{f+1}$ on a $(f+1)$-gon~\cite[Proposition 3]{PW93}, 
whose elements can be considered as the set of triangulations of a $(f+1)$-gon.
This enables them to then show that  the geometric realization 
of the poset $\primed^0_{f}$ is a topological sphere of dimension $f-3$.
For example, in Fig.~\ref{f:space}(iii) the poset $\mathcal T_{5}$ is 
pictured together with  the geometrical realization of the poset $\primed^0_{4}$, which is the simplicial complex consisting of ten 1-simplices as  pictured in Fig.~\ref{f:space} (ii). This 
complex is clearly homeomorphic to a 1-dimensional sphere. 

\subsection{The Penner-Waterman poset} 

By definition, diagrams corresponding to 
structures in the poset $\primed^r_{f}$  contain no arcs that pairwise intersect.
However, RNA molecules can fold into structures 
whose corresponding diagrams contain pairs of crossing arcs.
In this case, the RNA structure is no longer a secondary structure 
but what is commonly called an RNA {\em pseudoknot}.
RNA pseudoknots are wide-spread in nature and
have important functions~\cite{chen2009random}. We are
therefore interested in how to extend Penner and Waterman's analysis to these
more complicated structures.  

To this end, we focus on a special type 
of pseudoknot called a \emph{$k$-noncrossing  structure} 
(see, e.g. \cite{jin2008combinatorics}). This is an RNA structure 
whose diagram does not contain $(k+1)$ arcs that are pairwise mutually 
crossing. See Fig.~\ref{f:RNA:knot}(ii) for an example of $2$-noncrossing 
pseudoknot; among the three arcs in the pseudoknot, only arcs $(1,4)$ 
and $(5,9)$ do not form a crossing pair.
\rev{Note that $k$-noncrossing  structures 
	are also called $(k+1)$-noncrossing structures by some authors 
	(e.g. \cite{jin2008combinatorics}).}
RNA $k$-noncrossing structures have been studied extensively in the past decade
(see, e.g.~\cite{chen2009random}, and the reference therein), and 
they have interesting combinatorial properties~(see, e.g. \cite{reidys2010combinatorial}).

\begin{figure}[ht]
	\begin{center}
		\includegraphics[width=1\textwidth]{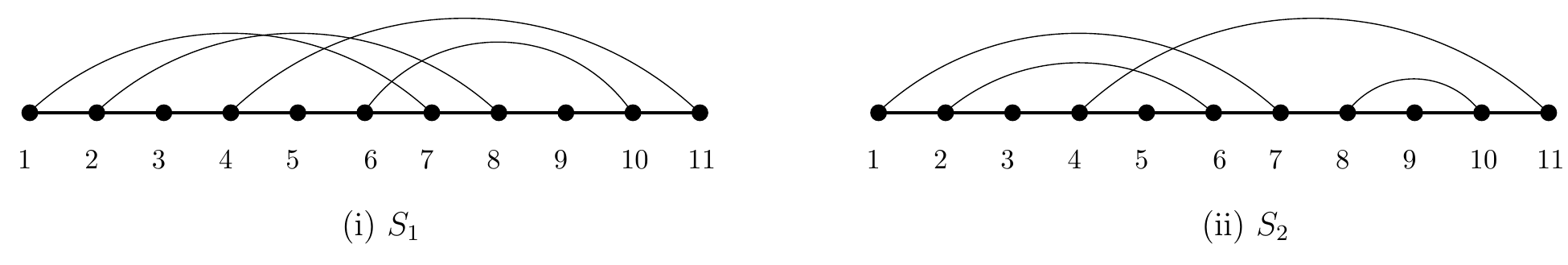}
	\end{center}
	\caption{ (i) A $3$-noncrossing diagram $S_1$ in $\properd^2_{3,3}$, which is not regular as $(1,7)$ and $(2,8)$ are two mutually
		crossing arcs both incident with the interval consisting of the non-free sites $1$ and $2$. 
		(ii) A  $2$-noncrossing diagram $S_2$ in $\primed^2_{3,2}$ which is regular and equivalent to $S_1$. }
	\label{f:regular:example}
\end{figure}

In this paper, we define and study a new poset  
$\primed^r_{f,k}$ ($k \ge 1$, $f \ge 3$, $r\ge 0$) of binary RNA $k$-noncrossing 
structures whose diagrams are {\em regular} and which contain 
exactly $f$ free sites and have tautological number at most $r$.
Intuitively, a regular diagram is one in which there are no 
two crossing arcs that are both incident with the same 
interval of non-free sites (see Fig.~\ref{f:regular:example}(ii) for 
an example and Section~\ref{sec:block:regular} for the precise definition).
We call the poset $\primed^r_{f,k}$ the {\em Penner-Waterman poset} 
since, as we shall see later, $\primed^r_{f,1} = \primed^r_{f}$.
%In addition to showing that $\primed^r_{f,k}$ is pure and of rank $k(2f-2k+1)+r-f-1$ (Theorem~\ref{mainresult2}), we prove that  the geometric realization of $\primed^0_{f,k}$ is the join of a simplicial sphere of dimension $k(f-2k)-1$ and an $\left((f+1)(k-1)-1\right)$-simplex (Theorem~\ref{mainresult1}). 

\rev{The two main results of this paper concern the Penner-Waterman poset. The first 
gives the structure of the geometric realization of $\primed^0_{f,k}$ 
(see Section~\ref{sec:preliminary} for the definition of this term).}

\begin{theorem}\label{mainresult1}
	For two integers $k\ge1$ and $f \ge 3$ with $f \ge 2k$, 
	the geometric realization of  $\primed^0_{f,k}$ is the join of 
	a simplicial sphere of dimension $k(f-2k)-1$ and an $\left((f+1)(k-1)-1\right)$-simplex.	
\end{theorem}

\rev{An illustration of Theorem~\ref{mainresult1} is presented in 
Fig.~\ref{f:generalspace} for a subposet of $\primed^0_{5,2}$ (the full poset is too large to include).
Our second main result concerns the structure of $\primed_{f,k}^r$.
Recall that rank of a (finite) poset is the length of a maximum length in the poset, 
and that the poset is pure in case all of its maximal chains have the same length.}

\begin{theorem}\label{mainresult2}
	%\label{thm:main}
	For any $r\ge 0$, $f \ge 3$ and $k\ge 1$ with $f\ge 2k$,  the  Penner-Waterman  poset $\primed_{f,k}^r$ 
	is pure and of rank $k(2f-2k+1)+r-f-1$. 
\end{theorem}

As a corollary, we also obtain an independent proof of 
Penner and Waterman's result mentioned above 
concerning the topology of $\primed^0_{f}$.

\begin{figure}[ht]
	\begin{center}
		\includegraphics[width=1\textwidth]{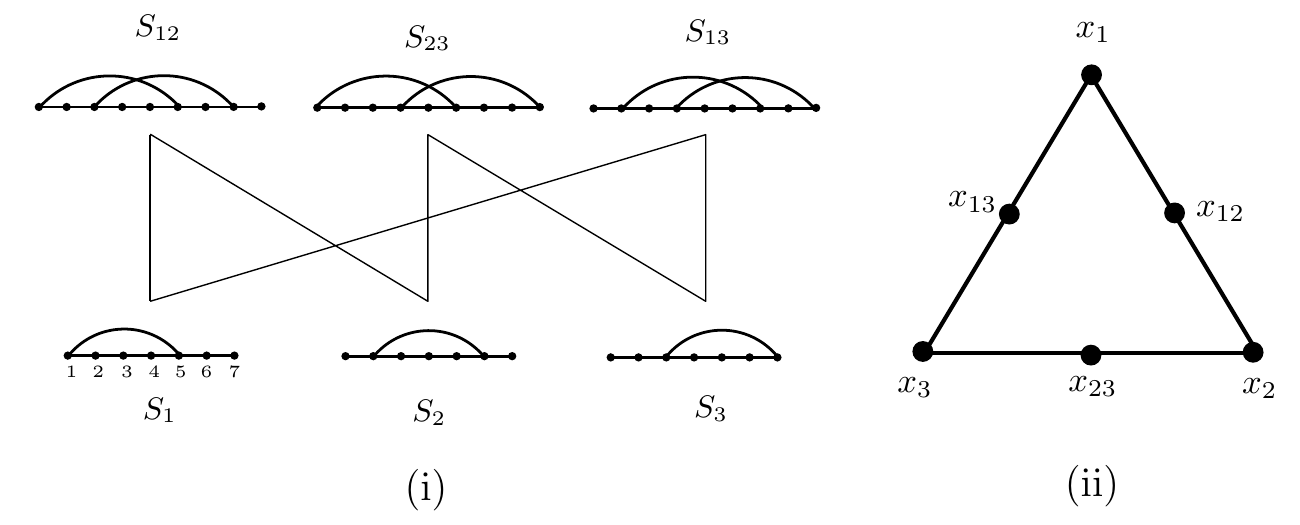}
	\end{center}
	\caption{ \rev{(i) The Hasse diagram of a subposet of $\primed^0_{5,2}$. (ii) The geometric realization of the 
		subposet in (i), which is a simplicial sphere of dimension one.  Here the vertices $x_i$ and $x_{ij}$ correspond to 
		RNA secondary structures $S_i$, $S_{ij}$ with the same indices.}}
	\label{f:generalspace}
\end{figure}

\subsection{Outline of the proof of Theorems~\ref{mainresult1} and~\ref{mainresult2}}
\label{subsec:outline}

As the proof of our main results is quite technical
and involved \rev{(see Section~\ref{sec:space})}, we now provide an overview to guide the reader through it. 
Our main tool is the introduction of an algebraic relationship
which associates a certain type of symmetric integral
matrix to $k$-noncrossing structure which we call its {\em block matrix}.  
Intuitively, these matrices encode the incidence relationship between blocks 
of an RNA structure, where a block is a maximal interval of non-free sites.
For instance, in the diagram $S_1$ in Fig.~\ref{f:regular:example} 
the interval consisting the of non-free sites $6$, $7$, and $8$ is a block; there 
are two arcs between this block and the block consisting of the 
non-free sites $1$ and $2$,  which is encoded as an 
entry with value $2$ in the associated block matrix. 
We took this approach since Penner and Waterman
used {\em rooted fattrees} to
prove their results but, even though fattrees can be generalized 
to fatgraphs (see e.g~\cite{penner2016moduli}), we could not
find a way to utilize this generalization in our proofs.

\bigskip 

\begin{figure}[h]
	\begin{center}
		\includegraphics[width=0.7\textwidth]{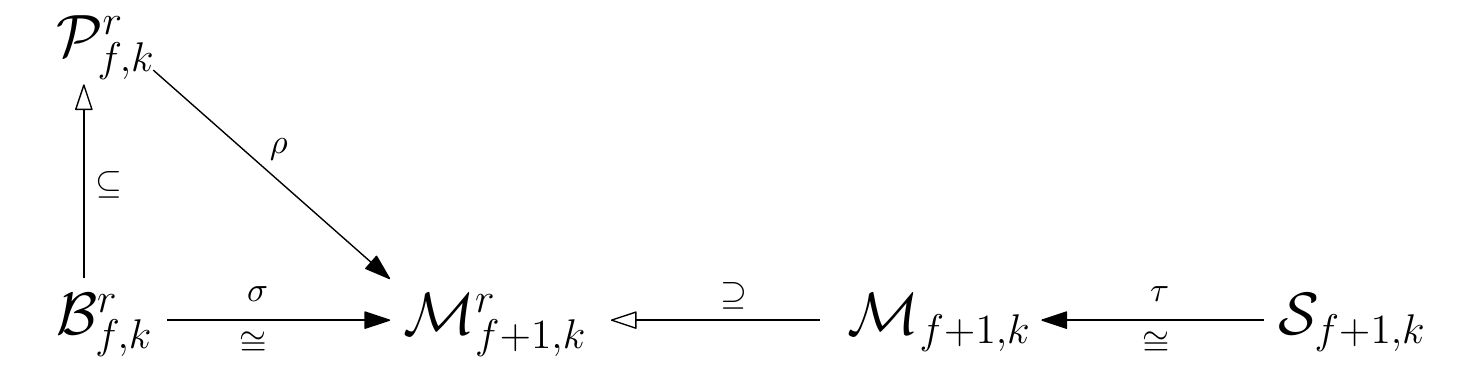}
	\end{center}
	\caption{\rev{An outline of the proof of the main results. Here  $\properd^r_{f,k}$ 
			is the poset of proper $k$-noncrossing diagrams 
			with $f$ free sites whose tautology number is at most $r$. The Penner-Waterman poset $\primed^r_{f,k}$ contains 
			all regular diagrams in $\properd^r_{f,k}$, and $\mathcal M^{r}_{f+1,k}$ 
			is the poset consisting of the block matrices of the diagrams in  $\primed^r_{f,k}$, 
			which are certain non-negative integral $k$-noncrossing matrices of order $f+1$. 
			Next,  $\mathcal S_{f+1,k}$ is the poset of all 
			(not necessarily binary) $k$-noncrossing diagrams 
			with $f+1$ free sites. Among the three poset homomorphisms, $\rho$ is surjective whilst 
		both $\sigma$ and $\tau$ are isomorphisms. }} \label{f:summary}
\end{figure}

Our proof proceeds as follows (see Fig.~\ref{f:summary}). 
The first goal is to show that $\primed^r_{f,k}$ 
is isomorphic to the poset  $\mathcal M^r_{f+1,k}$ of 
a certain family of symmetric integral matrices of order $f+1$ (Theorem~\ref{thm:poset:binary}). 
To this end, we introduce the set $\properd^r_{f,k}$ 
of {\em proper} diagrams,  a superset of $\primed^r_{f,k}$ consisting 
of binary diagrams in which each arc covers at least one, but not all 
of the free sites in the diagram (see Fig.~\ref{f:regular:example} for an example), and consider the map $\promap$ 
which takes each diagram in $\properd^r_{f,k}$ to  its block matrix.
To show that $\promap$ is surjective, we begin 
by defining an equivalence
relation $\sim$ on $\properd^r_{f,k}$, where $S\sim S'$ holds 
for two proper diagrams $S$ and $S'$ if and only if there exists a bijection 
between the arcs of $S$ to those of $S'$ that preserves the free 
sites below each arc.  We then show that two proper diagrams are equivalent if 
and only if they have the same block matrix~(Theorem~\ref{thm:equivalence}),
and that there exists a unique regular diagram within each equivalence class 
of $\sim$, namely the diagram with the minimum number of crossing 
arc pairs among all proper diagrams within that class (Theorem~\ref{thm:regular}). 
Using a characterisation of integral matrices that can 
be realized as the block matrix of a proper diagram~(Theorem~\ref{thm:surjective:matrix}), 
we then show that the collection of preimages $\{\promap^{-1}({\bf M}) \,:\, {\bf M} \in {\mathcal M}^r_{f+1,k}\}$ 
forms a partition of $\properd^r_{f,k}$, and that each preimage is precisely
one of the equivalence classes of $\sim$. From
this it follows that $\promap$ is surjective and, moreover, that its restriction 
to $\primed^r_{f,k}$, denoted by $\bmap$, is a poset isomorphism 
between $\primed^r_{f,k}$ and ${\mathcal M}^r_{f+1,k}$~(Theorem~\ref{thm:poset:binary}).

Our second goal is to show that $\primed^0_{f,k}$ is 
isomorphic to $\mathcal S_{f+1,k}$, the poset of all 
(not necessarily binary) $k$-noncrossing diagrams with $f+1$ free sites (Theorem~\ref{isomorphism}).
To do this, we show that mapping a diagram in $\mathcal S_{f+1,k}$ to its adjacency matrix gives a poset isomorphism $\tau$ from $\mathcal S_{f+1,k}$ 
to $\mathcal M^0_{f+1,k}$ (Proposition~\ref{prop:iso:rna}).
Composing $\bmap$ with the inverse of the map $\tau$ in case $r=0$ then gives a poset isomorphism 
between $\primed^0_{f,k}$ and $\mathcal S_{f+1,k}$.

The proof of our main results is then completed
by exploiting an interesting connection between $\mathcal S_{f+1,k}$ 
and a certain poset of {\em multitriangulations} of a polygon \cite{pilaud2009multitriangulations}, as 
detailed in Lemma~\ref{lem:top:sphere}.  
Since the topology of the poset of multitriangulations is well understood, this enables us 
to obtain the topology of $\primed^0_{f,k}$ in  Theorem~\ref{mainresult1}  and,
as a corollary, the fact that $\primed^r_{f,k}$ is pure in Theorem~\ref{mainresult2}.

\subsection{Organization of the rest of the paper} 

In Section~\ref{sec:preliminary} we collect together
some basic terminology and facts concerning diagrams and simplicial complexes.
We then introduce block matrices and regular diagrams in 
Section~\ref{sec:block:regular}, and the 
Penner-Waterman poset and $\mathcal M^r_{f+1,k}$ in Section~\ref{sec:pw:poset}. 
In Section~\ref{sec:equivalence} we define the equivalence 
relation $\sim$ on proper diagrams,  and in Section~\ref{sec:regular:rep} 
we show that there exists a unique regular diagram  within 
each equivalence class of $\sim$. In Section~\ref{sec:poset:iso} 
we show that the map $\bmap$ is a poset isomorphism between 
the Penner-Waterman  poset $\primed^r_{f,k}$ and the poset $\mathcal M^r_{f+1,k}$,
and that $\tau^{-1}\bmap$  gives a poset isomorphism
between $\primed^0_{f,k}$ and $\mathcal S_{f+1,k}$. 
Then in Section~\ref{sec:space} we prove our main results
by considering the above-mentioned
relationship between $\mathcal S_{f+1,k}$  and multitriangulations.
In the last section, we conclude with a brief discussion of  some possible future directions.

\section{Preliminaries}
\label{sec:preliminary}

\subsection{Diagrams}
\label{subsec:diagram}
In this paper, $n$ and $k$ are two positive integers with 
$n\ge 4$ and $k\ge1$, unless stated otherwise. 
% are two positive integers. 

Motivated by~\cite{hr15,PW93}, a {\em molecule diagram}, or just {\em diagram}, 
is a graph with vertex set $\{1,\dots,n\}$, where each vertex is 
called a {\em site}.  It consists of {\em base-pair} arcs, which are 
of the form $(s_1,s_2)$ with $1<s_2-s_1<n-1$, as well as $n-1$ backbone edges
which are of the form $\{i,i+1\}$ for $1\le i \le n-1$. 
In particular, the arc $(1,n)$ is not allowed in any diagram.
Note that the backbone edges play no part in our results, 
but we include them as it is usual to do so in the definition of RNA 
secondary structures. In particular, from 
now on all arcs that we consider will be base-pair arcs.
In figures of diagrams backbone edges are represented by horizontal lines, and 
base-pairs by semi-circles in the upper-half 
plane (see Fig.~\ref{f:diagram} for an example).
The {\em length} and {\em size} of a diagram are the number of vertices 
and arcs, respectively. A diagram is {\em trivial} if its size is zero, and {\em non-trivial} otherwise. 

Two sites $s$ and $s'$ are {\em adjacent} if $|s-s'|=1$ holds, and 
when $s<s'$, we denote the {\em interval} that contains all sites 
between $s$ and $s'$ by $[s,s']$.
In case $(s_1,s_2)$  is an arc we say that $s_1$ and $s_2$ are {\em base-paired}, 
and that $(s_1,s_2)$ is {\em supported} by $s_1$ and $s_2$.
Two arcs are {\em adjacent} if one of them is supported 
by a site $s$ and the other by a site adjacent to $s$.  
A site $s$ is called {\em free} if it does not support any arc, and it is {\em covered} by 
an arc $(s_1,s_2)$ if $s_1<s<s_2$ holds.  
An arc is called {\em degenerate} if it does not cover any free site, and {\em tiny} if it 
covers precisely one free site. 

\begin{figure}[ht]
	\begin{center}
		\includegraphics[width=0.9\textwidth]{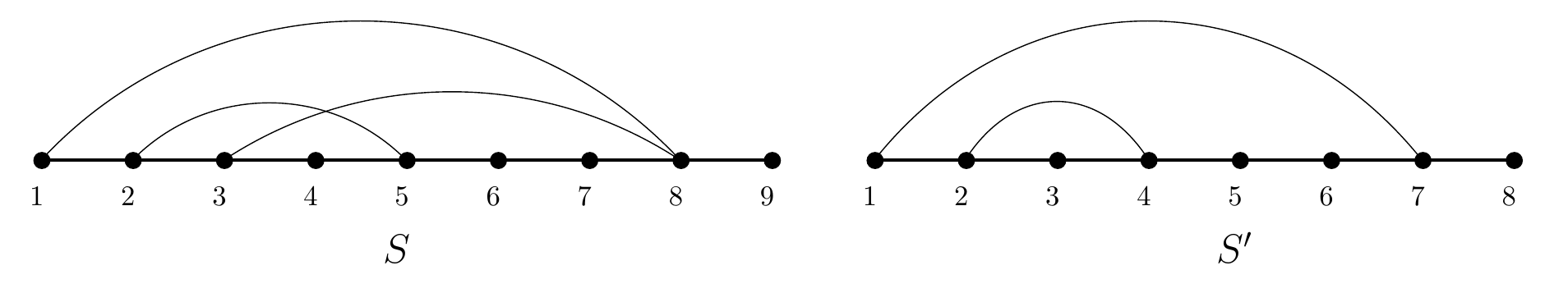}
	\end{center}
	\caption{Two RNA diagrams: $S$ is $2$-noncrossing  with length $9$ and size $3$, 
		and $S'$ is binary $1$-noncrossing.  Note that $S'$ is obtained from $S$ by suppressing the arc $(3,8)$.}
	\label{f:diagram}
\end{figure}

A diagram $S$ is called {\em binary} if it is non-trivial and each site supports at most one arc, and
it is called {\em proper} if it is binary and each 
arc covers at least one free site but no arc covers all of the free sites in $S$. Note that 
a proper diagram has a length of at least four and contains at least two free sites. 
Two arcs $(s_1,s_2)$ and $(s'_1,s'_2)$ in a diagram are {\em crossing} if 
either $s_1<s'_1<s_2<s'_2$ or $s'_1<s_1<s'_2<s_2$ holds.    
A diagram is a {\em $k$-noncrossing diagram} if it does not contain $(k+1)$ 
mutually crossing arcs, $k \ge 1$.  Note that binary $1$-noncrossing diagrams are 
the RNA secondary  structures defined in the introduction (see e.g.  \cite{PW93}). 
Also, the parameter $k$ in the definition of a $k$-noncrossing diagram 
is used in a way similar to that in a $k$-triangulation
(see Section~\ref{sec:space} and, e.g.~\cite{stump2011new}).
While some authors may refer to it as a  $(k+1)$-noncrossing 
diagram (see, e.g.,~\cite{chen2009random}), all  results in 
this paper can be easily adapted by shifting 
parameter $k$ to $k+1$. In Fig.~\ref{f:diagram} we present an example of a $2$-noncrossing diagram and a 
binary $1$-noncrossing diagram.   Two arcs are {\em parallel} if 
they cover the same set of free sites. The concept of 
parallel arcs as defined here is a natural 
generalization of that concept as defined in~\cite[p.~35]{PW93} since an 
arc $(s_1,s_2)$ in a binary $1$-noncrossing diagram is parallel 
to another arc $(s'_1,s'_2)$ if and only if $s_1$ and $s_2$ 
are adjacent to $s'_1$ and $s'_2$, respectively. 

Given an arc $(s_1,s_2)$ in a diagram $S$, we can derive a 
new diagram $S'$ with one less arc by applying one of the following two 
operations to $(s_1,s_2)$. The first operation is called {\em deleting} the arc, and $S'$ is obtained from $S$ 
by simply removing the arc $(s_1,s_2)$. In this case, the length of $S'$ is the same 
as that of $S$ while the number of free sites could  be the same, or increase by one or two. 
The second operation is called {\em suppressing} the arc, and  
$S'$ is obtained by removing $(s_1,s_2)$ from $S$ 
as well as removing any newly resulting free sites, and finally relabelling if necessary. 
Note that in this case the number of free sites in $S'$ is the same as that of $S$ 
while the length of $S'$ could be the same, or decrease by one or two. 
See Fig.~\ref{f:diagram} for an example, where $S'$ is obtained from $S$ 
by suppressing the arc $(3,8)$ and the length of $S'$ decreases by one 
since site $3$ is removed from $S$ while site $8$ is kept.
For technical reasons, we use the convention that a sequence of operations 
could be an empty sequence.

The set $\mathcal S_{n}$ of non-trivial diagrams with length $n$ has a natural poset structure under the  relation $\subseteq$ 
which is defined as 
follows:  $S \subseteq S'$ if $S$ and $S'$ have the same length, 
and each arc in $S$ is also an arc in $S'$. 
Note that $S\subseteq S'$ 
holds if and only if $S$ can be obtained from $S'$ by a sequence of arc deletions. 
We let $\mathcal S_{n,k}$ denote the subposet of $\mathcal S_{n}$ under $\subseteq$ 
that  consists of all possible $k$-noncrossing diagrams in $\mathcal S_n$. 
Note that $\mathcal S_{n,1}$ is precisely the poset of all 
secondary structures on linear molecules with length $n$ as studied in~\cite[p. 33]{PW93}.

\subsection{Simplicial complexes}

We now review some facts concerning posets and  
simplicial complexes that we will require later. More details can 
be found in~\cite{bjorner1995topological} and the references therein. 

Let $P=(P, \le)$ be a finite poset (partial ordered set). 
\rev{To ease the notation, for two elements $x_1,x_2$ in $P$ with  $x_1\le x_2$ and $x_1\not =x_2$, 
	we also write  $x_1<x_2$, or equivalently, $x_2>x_1$.  
}
A totally ordered subset $x_0<x_1<\cdots <x_t$ is called a {\em chain} of {\em length} $t$. 
The {\em rank} of $P$ is the maximum chain length taken over all chains in $P$.
If all maximal chains have the same finite length 
then $P$ is called {\em pure}.  For $x\in P$,  
the {\em open interval} $P_{>x}$ in $P$ is 
the set $\{y\in P\,:\,y > x\}$.
Suppose $Q=(Q,\preceq)$ is another poset. Then the {\em direct product} 
$P\times Q$ of two posets is the Cartesian 
product set ordered by $(x,y)\leq (x',y')$ if $x \le x'$ in $P$ and $y\preceq y'$ in $Q$. 
A  map $f:P \to Q$ is a {\em poset map} if it is {\em order-preserving}, 
that is,  $x\le y$ in $P$ implies $f(x)\preceq f(y)$ in $Q$. A {\em poset isomorphism} 
is a bijective poset map.  

A {\em simplicial complex} $\Delta$ on a finite vertex set $V$ is a 
collection of  nonempty  subsets of $V$, where each subset is called a {\em face} of $\Delta$, 
such that each nonempty subset of a face is also a face. 
The  {\em face poset} $\fp(\Delta)=(\Delta, \subseteq)$ of $\Delta$ is the 
set of faces in $\Delta$ ordered by inclusion.   
The {\em dimension} of a face is its cardinality (as a set) minus one, and 
the {\em dimension} of $\Delta$ is the size of a maximum face in $\Delta$. Note that, in line with 
usual conventions, the dimension of an empty complex is $-1$. 
A face whose dimension is the same as that of $\Delta$ is known as a {\em facet} of $\Delta$. 
Following~\cite{bjorner1995topological}, a $d$-dimensional simplicial complex is {\em pure} 
if every face is contained in a $d$-dimensional face. The complex consisting 
of all nonempty subsets of a $(d+1)$-element set is called the {\em $d$-simplex}.  
The geometric realization of a simplicial complex $\Delta$ is denoted by $|\Delta|$. 
A simplicial $d$-sphere is a simplicial complex whose geometric realization is homeomorphic to the $d$-dimensional sphere.

For two simplicial complexes $\Delta_1$ and $\Delta_2$ on two disjoint vertex sets, 
their {\em join} is the complex $\Delta_1 * \Delta_2 = \Delta_1 \cup \Delta_2 \cup \{F_1\cup F_2 \,|\, F_1 \in \Delta_1 ~\mbox{and}~F_2\in\Delta_2\}$. 
Given two nonempty spaces $X$ and $Y$,  their {\em join} $X*Y$ is the 
quotient space of $X\times Y \times I$ determined by the equivalence relation 
which identifies $(x,y_1,0)$ with $(x,y_2,0)$ and $(x_1,y,1)$ with $(x_2,y,1)$ 
for all $x,x_1,x_2$ in $X$ and $y,y_1,y_2$ in $Y$. The following relation between 
the join operation on complexes and that of topological spaces 
is well known  (see, e.g. Eq.(9.5) in~\cite{bjorner1995topological}) : 
\begin{equation}
|\Delta_1 * \Delta_2| \cong |\Delta_1|*|\Delta_2|
\end{equation}
where $\cong$ denotes homeomorphism. 

Given a poset $P$, its {\em order complex} $\Delta(P)$ is a 
simplicial complex whose vertices are the elements of $P$ and 
whose faces are the finite nonempty chains of $P$. 
We shall use the following important link between 
direct products of posets and the join operation 
(see Theorem 5.1 in~\cite{walker1988}, also~\cite{quillen1978}):
Given two posets $(P,\le)$ and $(Q,\preceq)$ and $x\in P, y\in Q$, we have
\begin{equation}
\label{eq:join:product}
|(P\times Q)_{>(x,y)} | \cong |P_{>x}| * |Q_{\succ y}|.
\end{equation}

\subsection{Adjacency matrices of diagrams}
\label{sec:nim}

%In this section,  we relate diagrams to nonnegative integral matrices. We begin by briefly describing some rudimentary facts concerning such matrices.

Let $m\ge 1$ be a positive integer. 
A {\em nonnegative integral symmetric} matrix $\mathbf{M}=(x_{i,j})$ of {\em order} $m$ 
is a square matrix  with  $m$ rows and $m$ columns such 
that each entry $x_{i,j}$ is a nonnegative integer, and $x_{i,j}=x_{j,i}$ holds for all $1\le i,j \le m$. If in addition 
each $x_{i,j}$ is either $0$ or $1$, then $\mathbf{M}$ is a symmetric $(0,1)$-matrix. 
Note that all matrices considered in this paper are symmetric and nonnegative integral. 
A matrix is called {\em trivial} if all its elements are zero, and {\em non-trivial} otherwise. 

For a  square matrix $\mathbf{M}=(x_{i,j})$ of order $m$, 
the {\em diagonal elements} in $\mathbf M$ are the entries $x_{i,i}$ with $1 \le i \le m$, and 
the  {\em semi-diagonal elements} are the superdiagonal entries $x_{i,i+1}$ with $1\le i < m$ and the subdiagonal entries $x_{i,i-1}$ with $1 < i \le  m$. Motivated by the terminology used in~\cite{hr15}, 
the two elements $x_{1,m}$ and $x_{m,1}$ are referred to as the {\em rainbow elements}. 
Now, for four distinct integers $a,b,c,d$ we say $\{a,b\}$ and $\{c,d\}$ are {\em crossing} if 
$
\min\{a,b\}<\min\{c,d\}<\max\{a,b\}<\max\{c,d\}
$
or
$\min\{c,d\}<\min\{a,b\}<\max\{c,d\}<\max\{a,b\},
$
two entries $x_{i,j}$ and $x_{p,q}$ in $\mathbf M$ are {\em crossing} if 
the index pairs $\{i,j\}$ and $\{p,q\}$ are crossing. 
If $\mathbf M$ does not contain  $(k+1)$ mutually crossing 
non-zero entries it will be referred to as a {\em $k$-noncrossing matrix}.

A matrix $A=(a_{i,j})$ is {\em dominated} by another matrix $B=(b_{i,j})$, written as $A\le B$, if they have the same order $m$ and $a_{i,j}\le b_{i,j}$  holds for $1\leq i,j \leq m$.
For $m\ge 4$, let $\mathcal M_m$ be the non-empty set of symmetric non-trivial $(0,1)$-matrices of order $m$ 
in which all diagonal, semi-diagonal and
rainbow elements are zero. Then $\mathcal M_m$ is a 
poset under the relationship $\le$, where $A \le B$ holds for matrices $A$ and $B$ in $\mathcal M_m$ if  $A$ is dominated by $B$. 
%a_{i,j}\le b_{i,j}$ holds for $1\leq i,j \leq m$. 
%In this case, we also say that $A$ is {\em dominated} by $B$. 
For $m\ge4$ and $k \ge 1$, we let $\mathcal M_{m,k}$ be the subset of $\mathcal M_m$ consisting of 
all $k$-noncrossing matrices. Then $(\mathcal M_{m,k},\le)$ is a subposet of $(\mathcal M_m, \le)$. 

%\subsection{Adjacency Matrices of Diagrams}
%\label{sec:adj}

Now, given a diagram $S$ with length $n$, we define its {\em adjacency matrix} $\am(S)$  
to be the symmetric $(0,1)$-matrix $(a_{i,j})$ of order $n$ such that $a_{i,j}$ is $1$ 
if either $(i,j)$ or $(j,i)$ is an arc. Note that both of the two rainbow elements 
of $\am(S)$ must be zero because the arc $(1,n)$ is not allowed in $S$ by definition. 
For example, the adjacency matrix $\am(S')=(a_{i,j})$ 
of the $1$-noncrossing diagram $S'$ in $\mathcal S_{8,1}$ that is depicted 
in Fig.~\ref{f:diagram} is the symmetric $(0,1)$-matrix of order $8$ 
whose non-zero entries are $a_{1,7},a_{7,1},a_{2,4}$ and $a_{4,2}$. 
This implies that $\am(S')$ is a matrix in $\mathcal M_{8,1}$.

As the last example indicates, a number of properties of a 
given diagram are determined by its adjacency matrix. For instance, 
the number of arcs in $S$ is half of the sum of the elements in $\am(S)$, and 
$S\subseteq S'$ if and only if $\am(S)$ is dominated by $\am(S')$.  
Moreover, the following result, whose proof is routine, shows that 
the map that associates a diagram with its adjacency matrix is in fact a poset isomorphism. 
% between $(\mathcal S_{n},\subseteq)$ and $(\mathcal M_{n}, \le)$.

\begin{proposition}
	\label{prop:iso:rna}
	For $n\ge 4$, the map $S\to \am(S)$ is a poset isomorphism between $(\mathcal S_n, \subseteq)$ and 
	$(\mathcal M_n, \le)$. In addition, its restriction to $\mathcal S_{n,k}$ with $k
	\ge 1$ is a poset 
	isomorphism between $(\mathcal S_{n,k}, \subseteq)$ and $(\mathcal M_{n,k}, \le)$.
	\eepf
\end{proposition}

For later use, we let $\tau$ denote the map from $(\mathcal S_{n,k}, \subseteq)$ to $(\mathcal M_{n,k}, \le)$ that associates a diagram with its adjacency matrix.

\section{Block Matrices and Regular Diagrams} 
\label{sec:block:regular}

In this section we introduce and study a matrix which can be associated 
to a diagram, called its block matrix, which 
contains information concerning the structure of the diagram 
in a more condensed form than its adjacency matrix.
Moreover, we show that for a special family of binary diagrams which we will call regular diagrams, they are $k$-noncrossing if and only if their block matrices are $k$-noncrossing.

Given a diagram $S$ with length $n$, denote the number of free sites 
in $S$ by $f =f_S \ge 0$.  If $f>0$, we let $u_1< \dots < u_f$ denote the set of free sites. 
Setting $u_0=0$ and $u_{f+1}=n+1$, then its {\em $i$-th block} $B_i=B_i(S)$ $(1\le i \le f+1)$ consists of all (necessarily non-free) sites $s$ in $S$ with $u_{i-1}<s<u_i$.
In other words, each maximal interval of non-free sites forms a block.
For example, the interval consisting of the non-free sites $1$, $2$, and $3$ 
is a block in the diagram in Fig.~\ref{f:block:example}. Indeed, 
the diagram contains four free sites and hence it has five blocks.
Note that a block can be the empty set, and that a diagram without any free sites  
has only one block. 
If an arc $\arc$ is supported by a site in a block $B$, then 
we say that $\arc$ and $B$ are {\em incident}. The {\em block list} 
of $S$, denoted by $\bb(S)$, is the list $(B_1,\dots,B_{f+1})$ consisting 
of all blocks of $S$ in the canonical order. The {\em block matrix} $\mathbf B(S)$ 
of $S$ is the symmetric matrix with order $(f+1)$ in which 
the $(i,j)$-entry is the number of arcs incident with 
both $B_i$ and $B_j$ (see Fig.~\ref{f:block:example} for an example).

\begin{figure}[ht]
	\begin{center}
		\includegraphics[width=1\textwidth]{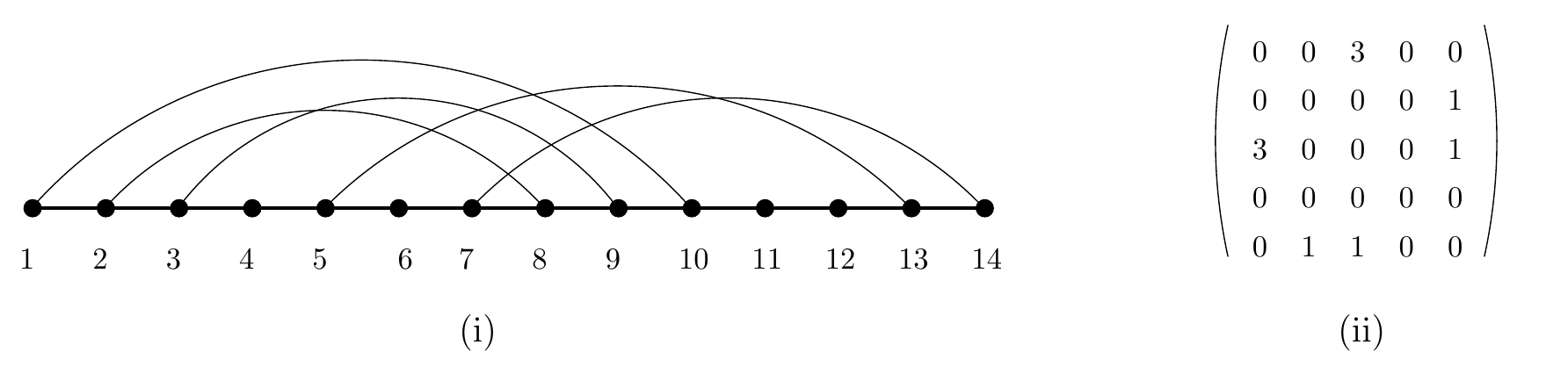}
	\end{center}
	\caption{ Example of a diagram and its block matrix: (i) a $4$-noncrossing diagram $S$; (ii) the block matrix ${\bf B}(S)$ associated with $S$. Since $S$ contains four free sites: $4$, $6$, $11$, and $12$, it has five blocks: $B_1=\{1,2,3\}$, $B_2=\{5\}$, $B_3=\{7,8,9,10\}$, $B_4=\emptyset$, and $B_5=\{13,14\}$.  }
	\label{f:block:example}
\end{figure}

The {\em number of crossings} in a diagram $S$ is 
the \rev{total} number of arc pairs in $S$ that are crossing. 
\rev{Note that a diagram that is not $k$-noncrossing contains a set of $k+1$ arcs that are pairwise mutually crossing, and hence its number of crossings is at least $k(k+1)/2$.}
A crossing of two arcs  $\arc_1$ and $\arc_2$ is called {\em local} 
if there exists a block $B$ of $S$ that contains two sites $s_1$ and $s_2$ 
such that $\arc_1$ and $\arc_2$ are supported by $s_1$ and $s_2$, 
respectively. 
A diagram is called {\em regular} if it is binary 
and does not contain any local crossings.
For example,  the diagram in Fig.~\ref{f:block:example} is not regular 
because $(2,8)$ and $(3,9)$ are two local crossing arcs. On the other hand, the diagram in Fig.~\ref{f:regular:example}(ii) is regular.

We now present two results relating properties of 
diagrams to those of their block matrices. The first one collects
together some facts and its straightforward proof follows from 
the relevant definitions.

\begin{lemma}
	\label{lem:incident:matrix}
	For a diagram $S$, the following statements hold:\\ 
	{\rm (i)}  $S$ does not contain any degenerate arc if and only if all diagonal elements in $\im(S)$ are zero. \\
	{\rm (ii)} $S$ does not contain any tiny arc if and only if all semi-diagonal elements in  $\im(S)$ are zero. \\
	{\rm (iii)} $S$ contains no parallel arcs if and only if $\im(S)$ is a $(0,1)$-matrix.   \\
	{\rm (iii)} If $S$ is binary, then $S$ is proper if and only if  all diagonal and rainbow elements in  $\im(S)$ are zero. \\
	{\rm (v)} If $S$ is proper and contains neither tiny nor parallel arcs, then $\im(S)$ is a  non-trivial $(0,1)$-matrix in which all diagonal, semi-diagonal and rainbow  elements are zero.  \eepf
\end{lemma}

The second result shows that the block matrix of a binary $k$-noncrossing diagram is $k$-noncrossing. 

\begin{lemma}
	\label{lem:sturcutre:matrix:noncrossing}
	The block matrix of a binary $k$-noncrossing diagram  is $k$-noncrossing. 
\end{lemma}

\begin{proof}
For simplicity, put $t=k+1$. Suppose that $S$ is a binary $k$-noncrossing diagram, 
and let $f$ be the number of free sites in $S$. Then the 
block matrix $\im(S)$ of $S$ can be written as $(b_{i,j})_{1\leq i, j \leq f+1}$. 
In addition, the list $\bb(S)$ of the blocks contained in $S$ can be written as $(B_1,\dots, B_{f+1})$.

We shall establish the lemma by contradiction. To this end, assume 
that $U=\{b_{r_1,c_1},\dots,b_{r_t,c_t} \}$  is a set of mutually crossing non-zero entries 
contained in $\im(S)$. We claim that $S$ 
contains a set of $t$ pairwise crossing arcs,  which contradicts 
the fact that $S$ is $k$-noncrossing.  

Since entries in $U$ are mutually crossing, it follows that for $1\le i < j \le t$, 
the integer pairs $\{r_i,c_i\}$ and $\{r_j,c_j\}$ are crossing. 
Therefore, $\{r_1,c_1,\dots, r_t,c_t\}$ is a set of $2t$ distinct integers.  
Since $\im(S)$ is symmetric, we may assume that $r_i<c_i$ holds for $1\le i \le t$. 
In addition, relabelling the indices if necessary, we may assume that $r_1<\cdots <r_t$. 

Now for each $1\le i \le t$, since $b_{r_i,c_i}>0$, 
we fix an arc $\arc_i=(s_i,s'_i)$ such that $s_i$ is contained in block $B_{r_i}$ 
and $s'_i$ is contained in block $B_{c_i}$.  
Then for $1\le i < j \le t$, because $b_{r_i,c_i}$ and $b_{r_j,c_j}$ are crossing in the matrix $\im(S)$, 
we have $r_i<r_j<c_i<c_j$, and hence $s_i<s_j<s'_i<s'_j$. This implies  $\arc_i$ 
crosses $\arc_j$, and hence $\{\arc_i\}_{1\le i \le t}$ is a 
set of $t$ mutually crossing arcs, as claimed.
\end{proof}

Note that the converse of Lemma~\ref{lem:sturcutre:matrix:noncrossing} 
does not hold in general. For instance,  
the diagram $S$ in Fig.~\ref{f:block:example} is $4$-noncrossing while its block matrix 
$\im(S)$ is $2$-noncrossing.  
However, the next result shows that the converse does hold for regular diagrams. 

\begin{proposition}
	\label{prop:noncrossing:matrix}
	Suppose that $S$ is a regular diagram.
	Then $S$ is $k$-noncrossing if and only if $\im(S)$ is a $k$-noncrossing matrix. 
\end{proposition}

\begin{proof}
 By Lemma~\ref{lem:sturcutre:matrix:noncrossing} it follows that 
if $S$ is $k$-noncrossing, then $\im(S)$ is $k$-noncrossing. 

To see that the converse holds, for simplicity, 
put $t=k+1$ and let $f$ be the number of free sites in $S$.  
Then the block matrix $\im(S)$ can be written as $(b_{i,j})_{1\le i, j \le f+1}$. In 
addition, the list $\bb(S)$ of the blocks contained in $S$ can be written as $(B_1,\dots, B_{f+1})$.

Now, suppose that $S$ contains a set $\{(s_i,s'_i)\}_{1\le i \le t}$ 
of mutually crossing arcs. 
Then it suffices to establish the claim that $\im(S)$
must have $t$ mutually crossing non-zero elements.

By swapping the indices if necessary, we may 
assume that $s_i<s_j$ holds for all $1\le i <j \le t$. 
For each non-free site $s$ in $S$, let $\alpha(s)$ 
be the index in $\{1,\dots,f+1\}$ such that $s$ is 
contained the block $B_{\alpha(s)}$ of $S$. Note that for 
two non-free sites $s<s'$, we have $\alpha(s) \le \alpha(s')$.  
In addition, if $(s,s')$ is an arc in $S$, then $b_{\alpha(s),\alpha(s')}\ge 1$. 

Now, consider the set $U=\{b_{\alpha(s_i),\alpha(s'_i)}\}_{1\le i \le t}$ 
of entries in $\im(S)$. Note first that each entry in $U$ 
is non-zero. Moreover each pair of distinct elements in $U$ are crossing. 
Indeed, fix two indices $i$ and $j$ with $1\le i < j\le t$. 
Since $(s_i,s'_i)$ and $(s_j,s'_j)$ are crossing and $s_i<s_j$, we have $s_i<s_j<s'_i<s'_j$, 
and hence $\alpha(s_i) \le \alpha(s_j) \le\alpha(s'_i)\le \alpha(s'_j)$. 
Using the fact that $S$ is regular, we can further conclude that
$$
\alpha(s_i) < \alpha(s_j) < \alpha(s'_i) < \alpha(s'_j),
$$
from which it follows that $\{\alpha(s_i),\alpha(s'_i)\}$ and $\{\alpha(s_j), \alpha(s'_j)\}$ 
are two pairs of crossing integers. Therefore $U$ consists 
of $t$ mutually crossing non-zero elements in $\im(S)$, 
which completes the proof of the claim. 
\end{proof}

\section{The Penner-Waterman Poset}
\label{sec:pw:poset}

In~\cite[p.\,35]{PW93} Penner and Waterman investigated the
poset $\primed^r_{f}$ of RNA secondary structures mentioned in the
introduction. In this section we introduce and study a 
generalisation of their poset for binary $k$-noncrossing diagrams. 

To define this new poset we need some additional terminology. 
Given a non-negative integral matrix $\mathbf M$ of order $m$, we let 
\begin{equation}
\mdr(\mathbf M)=\mdp(\mathbf M)+\mdq(\mathbf M), 
\end{equation}
be the {\em tautology number} of $\mathbf M$,  where
\begin{equation}
\label{eq:pq}
\mdp(\mathbf M)=\sum_{1\le i < j \le m } \max\{0, x_{i,j}-1\} ~~\mbox{and}~~\mdq(\mathbf M)=\sum_{1\le i <m} x_{i,i+1}.
\end{equation}
In addition, given a diagram $S$, we define its {\em tautology number} to be $\mdr(\im(S))$. 
In particular, if $S$ is $1$-noncrossing, then $\mdp(\im(S))$ and $\mdq(\im(S))$ are
the same as the values $p(S)$ and $q(S)$ as
defined in~\cite[p.35]{PW93}, respectively, and 
$\mdr(\im(S))$ is the number of tautological arcs in $S$. 
Moreover, $\mdr(\im(S))=0$ holds if and only if $S$ contains neither tiny nor parallel arcs. 
Finally, for $r\ge 0, m\ge4$, and $k\ge 1$, we let $\mathcal M^r_{m,k}$ be the non-empty set of symmetric non-trivial non-negative integral $k$-noncrossing matrices of order $m$ whose tautology number is less than or equal to $r$ and in which all diagonal and
rainbow elements are zero. 
Clearly, we have $\mathcal M^r_{m,k}\subseteq \mathcal M^{r+1}_{m,k}$ and $\mathcal M^r_{m,k}\subseteq \mathcal M^{r}_{m,k+1}$.
In particular, all matrices in $\mathcal M^0_{m,k}$ are necessarily $(0,1)$-matrices and hence we have $\mathcal M^0_{m,k}=\mathcal M_{m,k}$.

Now, we let  $\properd_{f,k}^r$ be the set of all proper $k$-noncrossing 
diagrams $S$ with $f$ free sites and tautology number $\mdr(\im(S))\le r$, 
and $\primed_{f,k}^r$ be the subset consisting of all 
regular diagrams in ${\properd}_{f,k}^r$.   
Note that for the special 
case $k=1$, we have $\properd_{f,1}^r = \primed_{f,1}^r$, which 
is the set $\mathcal B_{f}^r$ introduced in~\cite[p.35]{PW93}.
More generally, we have $ \primed_{f,k}^r \subseteq \properd_{f,k}^r$ in
view of the following observation.

\begin{lemma}
	\label{perfect:proper}
	A regular diagram is proper.
	%is necessarily proper. 
\end{lemma}

\begin{proof}
Suppose that $S$ is a regular diagram. Denote the length of $S$ by $n$ and the number of free sites in $S$ by $f$. Without loss of generality, we may assume that $n\ge 4$ and $f\ge 1$ as otherwise the lemma clearly holds. 

First we shall show that $S$ 
does not contain any degenerate arc. Suppose that this is not the case.
Then there exists a block $B$ of $S$ such that the set $\Sigma(B)$ of arcs $(s,s')$ in $S$ with $\{s,s'\} \subseteq B$ is not empty. Now fix an arc $\arc_1=(s_1,s'_1)$ in $\Sigma(B)$ so that the distance $|s_1-s'_1|$ is minimum over all arcs in $\Sigma(B)$. 
We claim that $s_1$ and $s'_1$ are adjacent, that is, $s'_1=s_1+1$,
which leads to a contradiction because no arc in $S$ is supported by two adjacent sites.

Suppose the claim is not true, and consider the site $s_2:=s_1+1$. Then $s_2$ is contained in $B$ 
with $s_1<s_2<s'_1$. Let $\arc_2$ be the arc supported by $s_2$ and 
denote the other site supporting $\arc_2$ by $s'_2$. 
Since $S$ is regular, $\arc_1$ and $\arc_2$ are not crossing, it follows that $s_1\le s'_2 \le s'_1$ and hence $\arc_2 \in \Sigma(B)$. Moreover, this implies that $|s_2-s'_2|<|s_1-s'_1|$, a 
contradiction.  Hence the claim holds, and therefore $S$ 
does not contain any degenerate arc.

Let $\Sigma(S)$ be the set of the arcs in $S$ that cover all free sites of $S$. It remains to prove that $\Sigma(S)$ is the empty set. 
Suppose this were not the case. Then 
fix an arc $\arc_3=(s_3,s'_3)$ in $\Sigma(S)$ so that  $s'_3-s_3 \ge s'-s$ holds 
for each arc $(s,s')$ in $\Sigma(S)$. 
Since $[s_3,s'_3]$ contains all free sites of $S$, it follows that $s_3$ is 
contained in $B_1(S)$, the first block of $S$, and that $s'_3$ is contained 
in $B_{f+1}(S)$, the last block of $S$. 
Moreover, neither $1$ nor $n$ is a free site, and thus site $1$ is contained in $B_1(S)$ and 
site $n$ is contained in $B_{f+1}(S)$. 

Denote the site that is base-paired with $1$ by $s^*$. 
Then $s^*$ is not contained in $B_1(S)$ as $S$ does not contain any 
degenerate arcs. In addition, we know that $s^*$ is contained in $B_{f+1}(S)$ 
and $s^*\ge s'_3$ as otherwise $(1,s^*)$ and $(s_3,s'_3)$ are two locally crossing 
arcs, a contradiction to the fact that $S$ is regular. This implies 
that $s^*-1\ge s'_3-s_3$, and hence $s^*-1= s'_3-s_3$ in view of the 
maximality of $(s_3,s'_3)$. It follows that $s_3=1$, and a similar argument 
shows that $s'_3=n$. Thus $(1,n)$ is an arc in $S$, a contradiction to the fact that 
$1<|s-s'|<n-1$ holds for every arc $(s,s')$ in a diagram of length $n$. 
Therefore, $\Sigma(S)$ is the empty set, which completes the proof.
\end{proof}

Next, we say $S \preceq S'$ holds for two diagrams $S$ and $S'$ in $\properd_{f,k}^r$  
if $S$ can be obtained from $S'$ by a sequence of arc suppressions.
Note that the relation $\preceq$ is distinct from the poset relation $\subseteq$ as defined 
in Section~\ref{sec:nim}. Moreover, it is straightforward to see 
that $S \preceq S'$ implies that $\im(S)$ is dominated by $\im(S')$.
We now show that $(\properd_{f,k}^r,\preceq)$ and  $(\primed_{f,k}^r,\preceq)$  are finite posets.

\begin{proposition}
	For $f\ge 2,r \ge 0,k\ge 1$ with $f+r\ge 3$, $(\properd_{f,k}^r,\preceq)$ is a non-empty finite poset. Furthermore, $(\primed_{f,k}^r,\preceq)$ is a subposet of
	$(\properd_{f,k}^r,\preceq)$.	
	%with $(\primed_{f,k}^r,\preceq)\subseteq (\properd_{f,k}^r,\preceq)$. 
\end{proposition}

\begin{proof}
When $r\ge 1$, it is straightforward to see that the set $\primed_{f,k}^r$, and hence also 
the set $\properd_{f,k}^r$, contains a diagram with one tiny arc for $f
\ge 2$ and $k\ge 1$. On the other hand, when $f\ge 3$, both sets contain a diagram with one arc covering precisely two free sites for $r\ge 0$ and $k\ge 1$. Therefore, we know that both sets are non-empty for  $f\ge 2,r \ge 0,k\ge 1$ with $f+r\ge 3$.

Since both sets are posets under the binary relation $\preceq$ 
and $\primed_{f,k}^r$  is a subset of $\properd_{f,k}^r$, 
it suffices to show that $\properd_{f,k}^r$ is a finite set. 
This  follows from the fact that the number 
of free sites $f$ and the length $n$ of  a proper diagram $S$ satisfy the following inequality: 
\begin{equation}
\label{eq:size:proper:diagram}
n \le { f+3 \choose 2}+ 2\mdp(\im(S)). 
\end{equation} 

To establish the inequality in Eq.~(\ref{eq:size:proper:diagram}), note 
first that after removing all but one arc from each parallel class of 
arcs, there exists at most one arc between each pair of blocks in $S$. 
Then the inequality follows because $S$ contains precisely $f$ free sites, the set of removed arcs contribute 
to at most $2\mdp(\im(S))$ non-free sites in $S$, and the remaining 
arcs contribute to at most ${f+1 \choose 2}$ non-free sites. 
\end{proof}

\begin{figure}[ht]
	\begin{center}
		\includegraphics[width=0.6\textwidth]{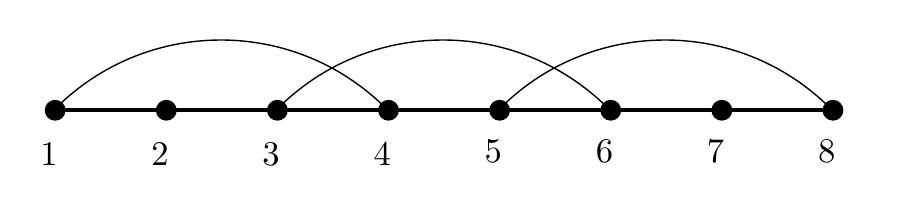}
	\end{center}
	\caption{A member of the infinite family of non-proper binary diagrams with two free sites
		defined in the text for $t=4$.  }
	\label{f:infinity}
\end{figure}

For $f\ge 3,r \ge 0,k\ge 1$, we call poset $(\primed_{f,k}^r,\preceq)$  the 
{\em Penner-Waterman} poset.  Note that it is important to consider 
proper diagrams in the definition of the set $\properd_{f,k}^r$. For example 
there exists infinitely many binary $3$-noncrossing diagrams that 
have two free sites and $\mdr(\im(S))=0$. In particular, 
for each positive integer $t\ge 1$, let $S_t$ be the 
diagram with length $4t$ that contains precisely all 
arcs of the form $(i,i+3)$ for every odd number $i$ with $1 \le i  \le 4t-3$ 
(see Fig.~\ref{f:infinity} for the diagram with $t=2$). Then $S_t$ 
contains precisely two free sites (the sites $2$ and $4t-1$), and it 
is $3$-noncrossing  with $\mdp(\im(S_t))=\mdr(\im(S_t))=0$.

\section{An Equivalence Relation on Diagrams}
\label{sec:equivalence}

In this section we present a characterization of the set of 
proper diagrams that have the same block matrix. 
In particular, we first define a certain equivalence 
relation  $\sim$ on the set of proper diagrams, and then show 
that two proper diagrams have the same block matrix
if and only if they are equivalent under $\sim$ (see Theorem~\ref{thm:equivalence}).

We begin by defining $\sim$.
Two proper diagrams $S$ and $S'$ are defined to be {\em equivalent}, denoted by $S \sim S'$, 
if they have the same length, and there exists a bijective map $\phi$ from 
the set of arcs in $S$ to the set of arcs in $S'$ that preserves free sites, that is, 
for each arc $\arc$ in $S$, a free site $s$ in $S$ is covered by $\arc$ 
if and only if $s$ is covered by $\phi(\arc)$ in $S'$. See Fig.~\ref{f:swap} for 
an example of two equivalent diagrams $S_1'$ and $S_2'$. Note that in this 
example there are two such bijective maps from the 
arc set of $S'_1$ to that of $S'_2$: one maps arc $e_1=(1,8)$ in $S_1'$ to arc $(1,8)$ in $S'_2$ 
and the other maps $e_1$ to arc $(2,6)$.
It is straightforward to check that $\sim$ is an equivalence relation on 
the set of proper diagrams, and that two equivalent diagrams have the same number of arcs. 

\begin{figure}[ht]
	\begin{center}
		\includegraphics[width=0.9\textwidth]{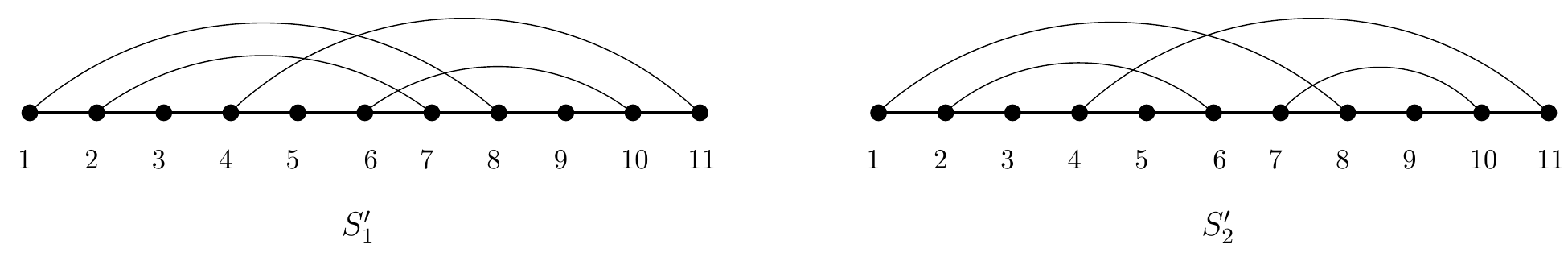}
	\end{center}
	\caption{Two equivalent proper diagrams $S'_1$ and $S'_2$ related by a single swap
		($S'_2$ can be obtained from $S'_1$ by a swap $\epsilon[6]$ at site $6$).
	}
	\label{f:swap}
\end{figure}

We now want to better understand when two proper diagrams are 
equivalent. To this end, we introduce a new operation on proper diagrams 
that preserves the $\sim$ relation. Suppose that  $s_1$ and $s_2=s_1+1$ 
are two adjacent non-free sites in a proper diagram $S$ of length $n$, and 
denote the site base-paired with $s_i$ by $r_i$ for $i=1,2$. A {\em swap} $\epsilon=\epsilon[s_1]$ at $s_1$ 
generates a binary diagram $\epsilon(S)$  by replacing the two arcs supporting by $s_1$ and $s_2$ 
with two new arcs: one is supported by $s_1$ and $r_2$, and the other by $s_2$ and $r_1$ 
(see Fig.~\ref{f:swap} for an example). 
Since $S$ is proper, we have
\begin{equation*}
\label{eq:swap:condition}
1<|s_1-r_2|<n-1~~\mbox{and}~~~1<|s_2-r_1|<n-1
\end{equation*}
and hence $\epsilon(S)$ is a binary digram. Note also that 
swapping is an involution, that is, swapping at $s_1$ twice results 
in the same diagram. We now show that $\epsilon(S)$ is a proper diagram
that is  equivalent to $S$, which has the same block matrix.

\begin{lemma}
	\label{lem:swap:properties}
	If $\epsilon$ is a swap on a proper diagram $S$, then 
	$\epsilon(S)$ is a proper diagram. Moreover, 
	$\epsilon(S)\sim S$,   $\bb(S)=\bb(\epsilon(S))$, and $\im(S)=\im(\epsilon(S))$. 
\end{lemma}

\begin{proof}
Denote the site at which $\epsilon$  swaps by $s_1$. Put $s_2=s_1+1$ 
and $S'=\epsilon(S)$. For $i=1,2$, let $\arc_i$ be the arc in $S$ 
supported by $s_i$ and denote the other site supporting $\arc_i$ by $r_i$. 
We shall assume that $\arc_1=(r_1,s_1)$ 
and $\arc_2=(s_2,r_2)$ since the other cases 
in which $\arc_1=(s_1,r_1)$ or $\arc_2=(r_2,s_2)$ (or both) can 
be established in a similar way. 

Let $\arc'_1=(r_1,s_2)$ and  $\arc'_2=(s_1,r_2)$. Then $S'$ 
is obtained from $S$ by replacing $\arc_1$ and $\arc_2$ 
by $\arc'_1$ and $\arc'_2$.
Note that the set of free sites covered by $\arc_1$ (i.e., those in $[r_1,s_1]$) 
is the same as the set covered by $\arc'_1$ (i.e., those in $[r_1,s_2]$). 
Similarly, the set of sites covered by $\arc_2$ is the same as 
that covered by $\arc'_2$.  Now let $\phi$ be the map 
that maps $\arc_i$ to $\arc'_i$ for $i=1,2$ and maps each of 
the other arcs in $S$ to the same arc in $S'$. Then $\phi$ is 
a bijection preserving free sites, and so $S\sim S'$.  

Since $S$ and $S'$ have the same length and the same set of free sites, it 
follows that $\bb(S)=\bb(S')$. In addition, for two arbitrary blocks $B_i(S)$ and $B_j(S)$ 
in $S$, the number of arcs between them is the same as that in $S'$, 
and so $\im(S)=\im(S')$.
\end{proof}

We now show that the swap operation can be used to convert a proper diagram 
into a canonical form.

\begin{lemma}
	\label{lem:move:arc}
	Suppose that $(s_1,s_2)$ is an arc in a proper diagram $S$, and $s'_i$ is a site 
	in the block containing $s_i$ for $i=1,2$.  Then there exists a sequence 
	of swaps $\epsilon_1,\dots,\epsilon_t$ for some $t\geq 0$ such 
	that $(s'_1,s'_2)$ is an arc in the  proper diagram 
	$\epsilon_t \epsilon_{t-1} \cdots \epsilon_1(S)$.
\end{lemma}

\begin{proof}
Let $B_i$ and $B_j$ be the block containing $s_1$ and $s_2$, respectively. 
Note that the lemma clearly holds when $s'_1=s_1$ and $s'_2=s_2$ (by taking
an empty sequence with $t=0$). Hence in the remainder of the proof we 
assume that either $s'_1 \not = s_1$ or $s'_2\not =s_2$.

Since $S$ is proper and $(s_1,s_2)$ is an arc, it follows $B_i\not = B_j$, 
and hence $|s^*_1-s^*_2|>1$ for each site $s^*_1$ in $B_i$ and $s^*_2$ in $B_j$. 
As the case $s_2 \not = s'_2$ can be established in a similar manner, we 
assume $s_1 \not = s'_1$. In addition, we further assume $s_1<s'_1$ as the proof 
of the other case  $s_1>s'_1$ is similar. 

We now proceed by induction on  $d=|s_1-s'_1|+|s_2-s'_2|$. 
For the base case $d=1$ we have  $s'_1=s_1+1$, and hence the 
lemma follows by taking $t=1$ and $\epsilon_1=\epsilon[s_1]$. 

Now suppose $d>1$ and the lemma holds for two arbitrary sites $s^*_1\in B_i$ and 
$s^*_2\in B_j$ with $0 < |s_1-s^*_1|+|s_2-s^*_2|<d$.  In particular, the lemma holds 
for  $s''_1=s'_1-1$ and  $s''_2=s'_2$. Therefore, there exists a sequence of 
swaps  $\epsilon_1,\dots,\epsilon_t$ for some $t\geq 1$ such that $(s''_1,s''_2)$ is 
an arc in the  proper diagram $S''=\epsilon_t \epsilon_{t-1} \cdots \epsilon_1(S)$. 
Since $|s''_1-s'_1|+|s''_2-s'_2|=1$, by the base case there exists a 
swap $\epsilon$ such that $(s'_1,s'_2)$ is an arc in the  proper 
diagram $S'=\epsilon(S'')$. This completes the proof of the induction step, 
and hence the lemma. 
\end{proof}

Using Lemma~\ref{lem:move:arc}, we now show that two equivalent 
proper diagrams have the same block list, which is a key step for 
establishing the main result in this section. %Theorem~\ref{thm:equivalence}. 
Note that the converse of this proposition clearly does not hold.

\begin{proposition}
	\label{prop:segment}
	If $S$ and $S'$ are two proper diagrams with $S\sim S'$, then $\bb(S)=\bb(S')$. 
\end{proposition}

\begin{proof}
Since $S \sim S'$, they have the same length, denoted by $n$, and 
also the same size, denoted by $m$. Thus it suffices to show 
that $S$ and $S'$ have the same set of free sites. 

We proceed by induction on $m$.  The base cases $m=1$ follows 
by noting that both $S$ and $S'$ contain precisely the same set of non-free sites with cardinality two.  Now assume $m>1$ and 
that two equivalent proper diagrams with at most $m-1$ arcs have the same set of free sites. 

Since $S$ contains two or more arcs, at least one block in the list $\bb(S)$ is non-empty. 
Now let $B_i$ be the first non-empty block in $\bb(S)$, and denote the 
largest site in $B_i$ by $s_1$. Let $s_2$ be the site so that $(s_1,s_2)$ 
is an arc in $S$ and denote the (necessarily non-empty) block containing $s_2$ by $B_j$. Since $S$ is proper, it follows that $j>i$. 
Let $s_3$ be the smallest site in $B_j$. 

Since $S$ and $S'$ are equivalent, there exists an arc $(s'_1,s'_2)$ 
in $S'$ such that the free sites in $S$ contained in $[s_1,s_2]$ are the same as those  in $S'$ 
contained in $[s'_1,s'_2]$. Let $B'_k$ and $B'_l$ be the blocks in $\bb(S')$ 
containing $s'_1$ and $s'_2$, respectively. Since $s_1+1$ is the smallest free 
site in $[s_1,s_2]$ and $[s'_1,s'_2]$ has the same set of free sites as 
that in $[s_1,s_2]$, it follows that the largest site in $B'_k$ is $s_1$. 
Similarly, the smallest site in $B'_l$ is $s_3$. 

By Lemma~\ref{lem:move:arc}, there exists a (possibly empty) sequence 
of swaps $\epsilon_1,\dots, \epsilon_t$ for some $t\ge 0$ such that $(s_1,s_3)$ 
is an arc in $S^*=\epsilon_t\cdots\epsilon_1(S)$. Similarly,  there exists a (possibly empty) 
sequence of swaps $\epsilon'_1,\dots, \epsilon'_{t'}$ for some $t'\ge 0$ such that $(s_1,s_3)$ 
is an arc in $S''=\epsilon_{t'}\cdots\epsilon_1(S')$. By Lemma~\ref{lem:swap:properties} 
it follows that $S^*$ and $S''$ are two equivalent proper diagrams. 
Consider the diagrams ${S}_o^*$ and ${S}_o''$ 
that are obtained from $S^*$ and $S''$ respectively by suppressing 
the arc $(s_1,s_3)$. Then ${S}_o^*$ and ${S}_o''$ are two equivalent proper 
diagrams with precisely $m-1$ arcs. Now the induction assumption implies 
that  ${S}_o^*$ and ${S}_o''$ have the same set of free sites, from 
which it follows that $S$ and $S'$ have the same set of free sites. This completes the 
proof of the induction step and hence the proposition. 
\end{proof}

The next result shows that the blocks of a binary diagram are determined by its block matrix (note that the converse clearly does not hold).

\begin{lemma}
	\label{lem:characteristic:matrix}
	Suppose that $S$ and $S'$ are two binary diagrams with $
	\im(S)=\im(S')$. Then $S$ and $S'$ have the same size and the same length.  
	Furthermore, we have $\bb(S) = \bb(S')$.
\end{lemma}

\begin{proof}
Let $(b_{i,j})_{1\le i,j\le (f+1)}$ be the elements 
in the block matrix $\im=\im(S)=\im(S')$, where $f+1$ is the order of $\im$. Then both $S$ and $S'$ contain $f$ free sites. Next, the size of $S$ is
$\sum_{1\leq i \le j \leq f+1} b_{i,j}$, the same as that of $S'$. Note that 
for a binary diagram, the number of its size is half of the difference 
between its length and the number of its free sites. Thus $S$ and $S'$ also have the same length. 
Finally, let $(B_1,\dots,B_{f+1})$ and 	$(B'_1,\dots,B'_{f+1})$ be the block lists of $\bb(S)$ and $\bb(S')$, respectively. Then for $1\le i \le f+1$, we have 
\begin{equation*} 
%\label{eq:block:size}
|B_i|= b_{i,i}+\sum_{1\le j \le f+1} b_{i,j}=|B'_i|,
\end{equation*}
from which $\bb(S) = \bb(S')$ follows.
\end{proof}

We now prove the main result of this section. 

\begin{theorem}
	\label{thm:equivalence}
	Suppose that $S$ and $S'$ are two proper diagrams. Then 
	\begin{equation*}
	\im(S)=\im(S')~~\mbox{if and only if}~~~S\sim S'. 
	\end{equation*}
\end{theorem}
\begin{proof}
Suppose that $S$ and $S'$ are two proper diagrams with $\im(S)=\im(S')$.
Let $(b_{i,j})_{1\le i,j\le (f+1)}$ be the elements in the block matrix $\im=\im(S)=\im(S')$, 
where $f+1$ is the order of $\im$. Denote the set of indices $(i,j)$ in $\im$ with $b_{i,j}>0$ and $i\le j$ by $U$.

By Lemma~\ref{lem:characteristic:matrix}, $S$ and $S'$ have the same number of arcs, denoted by $m$, and the same length.  Denoting the $f$ free sites of $S$ by $s_1<\cdots<s_{f}$,  then using Lemma~\ref{lem:characteristic:matrix} again we know that the free sites of $S'$ are also $s_1,\dots, s_{f}$.

Now, the set $U$ induces a partition of the arcs in $S$ as follows. 
For each $(i,j)$ in $U$, let $A_{i,j}$ be the set of arcs in $S$ between 
blocks $B_i(S)$ and $B_j(S)$. Then $A_{i,j}$ and $A_{i',j'}$ are disjoint 
for two distinct pairs $(i,j)$ and $(i',j')$ in $U$, and the 
set of arcs in $S$ is the disjoint union $A_{i,j}$ over $U$ in 
view of  $\sum_{(i,j)\in U}b_{i,j} =m$. In other words, $\{A_{i,j}\}_{(i,j)\in U}$ 
is a partition of the set of arcs in $S$.  Similarly, let $A'_{i,j}$ be the set of arcs in 
$S'$ between blocks $B_i(S')$ and $B_j(S')$ so that $\{A'_{i,j}\}_{(i,j)\in U}$ is a 
partition of the set of arcs in $S'$. 

For every $(i,j)$ in $U$, and two arcs $(r,s)$ in $A_{i,j}$ and $(r',s')$ in $A'_{i,j}$, 
the set of free sites in $[r,s]$ is $\{s_{i+1},s_{i+2},\dots,s_{j-1}\}$, 
which is also  the set of free sites in $[r',s']$. 
Together with $|A_{i,j}|=a_{i,j}=|A'_{i,j}|$, 
this implies that we can fix a bijection $\phi_{i,j}$ from $A_{i,j}$ to $A'_{i,j}$
that preserves free sites. Now let $\phi$ be the map from the 
arcs in $S$ to those in  $S'$ such that $\phi$ is $\phi_{i,j}$ 
when restricted to $A_{i,j}$ for each $(i,j)$ in $U$. Then 
$\phi$ is a bijection between the arcs in $S$ and those in $S'$ that preserves free sites. 
Hence $S\sim S'$.\\

Conversely,  suppose $S\sim S'$. Denote the number of free sites 
in $S$ by $f$. Then we have $f\ge 1$ since $S$ is proper. 
By Proposition~\ref{prop:segment}, diagrams $S$ and $S'$ have the same 
block list, and hence also the same set of free sites, which we enumerate 
by $s_1<s_2<\cdots <s_f$. Thus $\im(S)$ and $\im(S')$ have the same order $f+1$. 

Let $(b_{i,j})=\im(S)$ and $(b'_{i,j})=\im(S')$. By 
Lemma~\ref{lem:incident:matrix} it follows that $b_{i,i}=b'_{i,i}=0$ for $1\le i \le f+1$. 
Now fix an arbitrary pair of indices $1\le i < j \le f+1$. It suffices 
to show $b_{i,j}=b'_{i,j}$. To see this, note that $b_{i,j}$ is 
the number of arcs in $S$ between block $B_i(S)$ and $B_j(S)$, which is  
precisely the number of arcs $(s_1,s_2)$ in $S$ such that 
the set of free sites in $[s_1,s_2]$ is $\{s_i,\dots,s_{j-1}\}$. 
Since $S\sim S'$, there are precisely $b_{i,j}$ arcs $(s'_1,s'_2)$ 
in $S'$ such that the set of free sites in $[s'_1,s'_2]$ is $\{s_i,\dots,s_{j-1}\}$. 
This implies $b_{i,j}=b'_{i,j}$, from which $\im(S)=\im(S')$ follows.
\end{proof}

\section{Canonical Representatives}
\label{sec:regular:rep}

In this section, we show that there exists a unique regular diagram 
within each equivalence class of the 
equivalence relation $\sim$ on proper diagrams, and 
that this regular diagram has the minimum number 
of crossing arc pairs among all diagrams within that equivalence 
class (see Theorem~\ref{thm:regular}).

To this end, first note that for a swap  $\epsilon$ on a binary diagram $S$,  
$S$ and $\epsilon(S)$ have the same number of non-local crossings. Moreover,  
the number of local crossings in $\epsilon(S)$ either increases or decreases by one. 
We therefore call a swap $\epsilon$ is called {\em strict} if 
$\epsilon(S)$ has one less crossing than $S$, or equivalently, $\epsilon(S)$ has 
one less local crossing than $S$. 
For example, the swap illustrated in Fig.~\ref{f:swap} is strict.  	

Next, restricting the definition of  regular, we call 
a block $B$ in a binary diagram $S$ regular if $B$ does not contain two sites $s_1$ and $s_2$
for which there exists a pair of crossing arcs $\arc_1$ and $\arc_2$ in $S$ that are supported by $s_1$ and $s_2$, respectively.  
In particular, a binary diagram $S$ is regular if and only if every block in $S$ is regular.
The following result 
shows that blocks in proper diagrams that are not regular contain 
some special structures, which can be regarded as the `obstruction' to their being regular.

\begin{lemma}
	\label{adjacent:non-perfect}
	A non-regular block in a proper diagram is incident with two adjacent  arcs that are crossing. 
\end{lemma}

\begin{proof}
Suppose that $B$ is a non-regular block in a binary diagram $S$ and let $b$ be 
the number of sites in $B$, so that $b \ge 2$. We establish the lemma by using induction on $b$. 

The base case is $b=2$, is straightforward to check. 
So assume that $b> 2$ and the lemma holds for all 
non-regular blocks containing at most $b-1$ sites. 

Denote the sites in $B$ by  $s_1 < \dots <s_b$, enumerated from 
the smallest to the largest. Moreover, for $1 \le i \le b$, denote the 
site that is base-paired with $s_i$ by $s'_i$, and let $\arc_i$ be the arc 
supported by $s_i$ and $s'_i$.  Let $\Sigma$ be the set of arcs which 
are supported by at least one site in $B$.  Consider the arc $\arc$ 
supported by $s_b$ and put $\Sigma^*:=\Sigma\setminus \{\arc\}$. 
Then we have the following two cases:\\

\noindent
{\bf Case I:}  The arc set $\Sigma^*$ does not contain a pair of crossing arcs.\\

By assumption,  it follows that $\arc_b$ crosses an arc $\arc_i$ for some $1\le i \le b-1$. 
Now, assume $s_{b-1}<s'_{b-1}$. Then the interval $[s_{b-1},s'_{b-1}]$ 
contains neither $s_i$ nor $s'_i$ for $1\le i < b-1$. It follows that $s'_b$ 
is not contained in $[s_{b-1},s'_{b-1}]$ as otherwise $\arc_b$ 
and $\arc_i$ are not crossing for $1\le i \le b-1$, a contradiction. 
This implies that $\arc_b$ is crossing with $\arc_{b-1}$.

So assume $s_{b-1}>s'_{b-1}$. Then the  interval $[s'_{b-1},s_{b-1}]$ 
contains both $s_i$ and $s'_i$ for $1\le i < b-1$. It follows that $s'_b$ 
is contained in $[s'_{b-1},s_{b-1}]$ as otherwise $\arc_b$ and $\arc_i$ 
are not crossing for $1\le i \le b-1$, a contradiction. 
This again implies that $\arc_b$ is crossing with $\arc_{b-1}$.

The induction step now follows immediately (note that the induction assumption 
is not required for establishing this case).\\

\noindent
{\bf Case II:}  The arc set $\Sigma^*$ contains a 
pair of crossing arcs. \\

Consider the proper diagram $S'$ obtained from $S$ by 
suppressing the arc $\arc_b$. Since $S$ is proper, there is a 
block $B'$ in $S'$ that contains all the sites in $B$ other than $s_b$. 
Denote the set of arcs which are supported by some site in $B'$ by $\Sigma'$. 
Then $B'$ is a non-regular block in $S'$ with precisely $b-1$ sites. Thus 
by the induction assumption $B'$ is incident with a pair of adjacent 
crossing arcs $\arc$ and $\arc'$ in $S'$. Since $\arc$ and $\arc'$ 
are obtained from two adjacent crossing arcs in $S$ that are 
incident with $B$,  the induction step follows.
\end{proof}

\noindent 
{\bf Remark:} Using a more detailed  analysis in Case II, the above argument 
can be extended to show that the last lemma also holds for binary diagrams. \\

The last result enables us to 
show that each proper diagram can be converted 
into a regular diagram within its equivalence class by applying 
a sequence of strict swaps. 

\begin{proposition}
	\label{prop:swap}
	Given a proper diagram $S$, there exists a sequence of strict 
	swaps $\epsilon_1,\dots,\epsilon_t$ for some $t\geq 0$ such 
	that  $\epsilon_t \epsilon_{t-1} \cdots \epsilon_1(S)$ is a regular diagram that is equivalent to $S$.
\end{proposition}

\begin{proof}  If $S$ is regular, the proposition holds by taking $t=0$. 
Hence we may assume in the remainder of the proof that $S$ 
is not regular. Let $b$ be the number of arc pairs in $S$ that 
are locally crossing. Then we have $b\ge 1$ as $S$ is not regular. 
We now establish the proposition by using induction on $b$. 

For the base case $b=1$, which implies that there exists a
non-regular block $B$ in $S$. By Lemma~\ref{adjacent:non-perfect}, there 
exist two adjacent sites $s_1$ and $s_2=s_1+1$ in $B$ 
so that the arc $\arc_1$ supported by $s_1$ is crossing 
with the arc $\arc_2$ supported by $s_2$. Since $S$ is proper,  
by Lemma~\ref{lem:swap:properties} we can apply a swap $\epsilon$ at $s_1$ to obtain a 
proper diagram $S'=\epsilon(S)$ with $S'\sim S$. Since $\arc_1$ and $\arc_2$ 
are crossing, 
it is straightforward to check that diagram $S'$ contains one 
less local crossing than that of $S$, and hence $S'$ is 
regular, completing the proof of the base case. 

Now assume that $b>1$ and the lemma holds for each 
proper diagram $S$ which contains at most $b-1$ local crossing arc pairs. 
An argument similar to the base case shows that 
there exists a strict swap $\epsilon_1$ so that $\epsilon_1(S)$ 
is a proper diagram which is equivalent to $S$ and which has $b-1$ local crossing arc pairs. 
The proposition now follows by induction. 
\end{proof}

To illustrate the last result, note that the proper diagram $S_1$  
in Fig.~\ref{f:regular:example} can be converted to the regular 
diagram $S_2$ in Fig.~\ref{f:regular:example} by three strict swaps, 
consecutively acting on sites $1$, $6$ and $7$ (see Fig.~\ref{f:swap} where $\epsilon[1](S_1)$ 
and $\epsilon[6]\epsilon[1](S_1)$ are depicted as $S'_1$ and $S'_2$, 
respectively). 

Now we show that regular diagrams are unique within any equivalence class of $\sim$. 

\begin{lemma}
	\label{lem:regular:unique}
	If $S$ and $S'$ are two regular diagrams with $S\sim S'$, then $S=S'$.
\end{lemma}

\begin{proof}
Let $S$ and $S'$ be two regular diagrams with $S\sim S'$. Then 
they have the same length, denoted by $n$, and 
they contain the same number of arcs, denoted by $m$.  Without loss of generality, 
we may assume that $n\ge 4$ as otherwise the lemma clearly holds.

By Proposition~\ref{prop:segment} we have $\bb(S')=\bb(S)$ (which we
shall denote by $\bb$), and it contains $f+1$ blocks for some $f\ge 1$ as $S$ 
and $S'$ are both proper in view of Lemma~\ref{perfect:proper}. 
Note that this implies that $S$ and $S'$ have the same set of 
free sites. Since $S\sim S'$, we can fix a bijection $\varphi$ that 
maps each arc $\arc$ in $S$ to an arc $\varphi(\arc)$ 
such that a free site $s$ in $S$ is covered by $\arc$ if 
and only if $s$ is covered by $\varphi(\arc)$ in $S'$. 

We now prove the lemma by induction on $m$. The 
base case is $m=1$. Since $S$ and $S'$ have length $n$ and have the same set of two non-free sites, it follows that $S=S'$ and hence the lemma follows. For 
induction step, assume $m > 1$ and the lemma holds for 
any two equivalent regular diagrams with at most $m-1$ arcs. 

Since $m> 1$, there exist at least two non-free sites in $S$. 
Denote the first non-free site in $S$ by $s_1$, and let $s_2$ be the 
site in $S$ so that $(s_1,s_2)$ is an arc in $S$. 
Since $S$ and $S'$ have the same set of free sites, it follows 
that $s_1$ is also the first non-free site in $S'$. 
Let $s'_2$ be the site in $S'$ so that $(s_1,s'_2)$ 
is an arc in $S'$. Denote the block containing $s_1$ by $B_i$ 
and the one containing $s_2$ by $B_j$. 
Then $1\le i < j \le f+1$ since $S$ is proper and $B_i$ is the 
first non-empty block in $\bb$. Moreover, 
if $(s''_1,s''_2)$ is the arc in $S'$ that is the 
image of $(s_1,s_2)$ under $\varphi$, then $s''_1$ 
is contained in $B_i$ and $s''_2$ is contained in $B_j$ 
because $[s_1,s_2]$ and $[s''_1,s''_2]$ contain the same set of free sites
(see Fig.~\ref{f:unqiue:proof} for an illustration of the notation).

\begin{figure}[ht]
	\begin{center}
		\includegraphics[width=0.9\textwidth]{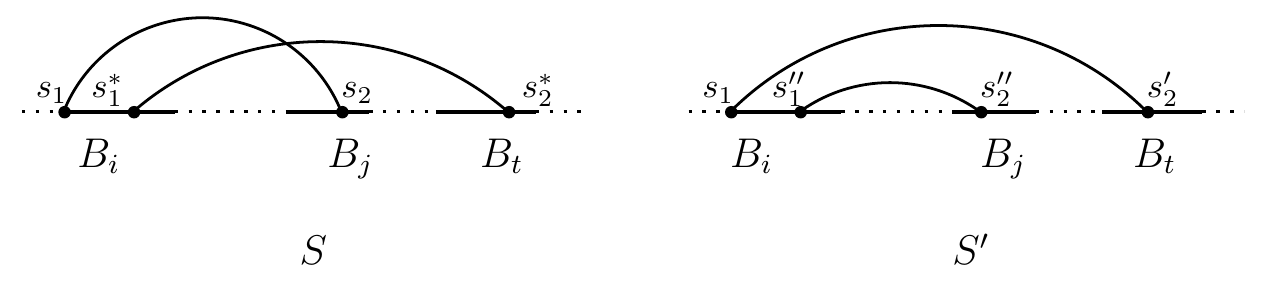}
	\end{center}
	\caption{An illustration for one step in the proof Lemma~\ref{lem:regular:unique}. Note that $S$ and $S'$ have the same set of blocks, among which only three are shown here.}
	\label{f:unqiue:proof}
\end{figure}

Next we shall show that $s'_2$ is contained in $B_j$. 
Note that if $s_1=s''_1$, then we have $s''_2=s'_2$, and 
hence $s'_2$ is contained in $B_j$. Therefore we only 
need to consider the case  $s_1\not =s''_1$. This 
implies that $s_1<s''_1<s'_2$ holds because $s_1$ 
is the first non-free site and $B_i$ is the first 
non-empty block which contains $s_1$ and $s''_1$ 
but not $s_2$. In addition, let $B_t$ be the block in $\bb$ 
that contains $s'_2$. Then we have $i<t\le f+1$ and so it remains to show $t=j$.  

To this end, first note that if $t<j$, then $s_1<s''_1<s'_2<s''_2$ and hence 
$(s_1,s'_2)$ and $(s''_1,s''_2)$ are two locally crossing arcs in $S'$, 
a contradiction to the fact that $S'$ is regular.  On the 
other hand, if $j<t$, then consider the arc $(s^*_1,s^*_2)$ 
in $S$ such that $(s^*_1,s^*_2)$ is mapped to $(s_1,s'_2)$ 
by $\varphi$. Since $[s^*_1,s^*_2]$ and  $[s_1,s'_2]$ contain the 
same set of free sites, it follows that $s^*_1$ is contained in 
block $B_i$ and $s^*_2$ is contained in block $B_t$. 
Note that we may further assume that $s_1<s^*_1$ 
because otherwise we have $s_1=s^*_1$, and hence $s_2=s^*_2$, from 
which $t=j$ follows. Together with $j<t$, this implies $s_1<s^*_1<s_2<s^*_2$, and thus 
$(s_1,s_2)$ and $(s^*_1,s^*_2)$ are two locally crossing arcs in $S$, 
a contradiction to the fact that $S$ is regular (see Fig.~\ref{f:unqiue:proof} for an illustration of this case).  
Therefore, we have $t=j$ and hence $s'_2$ is contained in $B_j$.

Our next step is to show that $(s_1,s_2)$ is an arc contained in $S'$. 
To this end, it suffices to show $s'_2=s_2$. Denote the number of 
sites in $B_j$ that are smaller than  $s_2$ by $b$, and that are smaller than $s'_2$ by $b'$.
Moreover, let $r+1$ be the smallest site in $B_j$, that is,  $B_j$ 
contains $r+1$ but does not contain $r$.  Note that for each site $s_3$ 
in block $B_j$ with $s_3>s_2$, there exists a site $s^*_3>s_3$ 
so that $(s_3,s^*_3)$ is an arc in $S$ since otherwise
$(s_1,s_2)$ and the arc supported by $s_3$ are local crossing, 
a contradiction. Therefore, the number of arcs $(s,s^*)$ in 
$S$ such that $r$ is the largest free site in $[s,s^*]$ is $b+1$. 
A similar argument shows that the number of arcs $(s',s'')$ 
in $S'$ so that $r$ is the largest free site in $[s',s'']$ is $b'+1$. 
Since $S$ is equivalent to $S'$, it follows that $b+1=b'+1$ holds, and hence $s'_2=s_2$. 

Noting that $(s_1,s_2)$ is a common arc in $S$ and $S'$, we 
finally consider the two diagrams $S_o$ and $S_o'$ that are 
obtained  from $S$ and $S'$ respectively by suppressing the arc $(s_1,s_2)$. 
Then $S_o$ and $S'_o$ are  two equivalent regular diagrams with $m-1$ arcs. 
By the induction assumption we have  $S_o=S'_o$, and thus $S=S'$. This completes
the proof of the induction step, from which the theorem follows. 
\end{proof}

We now prove the main result of this section. 

\begin{theorem}
	\label{thm:regular}
	Given a proper diagram $S$,  there exists a unique regular diagram  $S'$ 
	with $S\sim S'$. 
	% that is, $\im(S)=\im(S')$. %n $\mathcal S$. 
	Moreover, $S$ is regular if and only if the number of crossings in $S$ has 
	the minimal number of crossings in its equivalence class.
\end{theorem}
\begin{proof}
By Proposition~\ref{prop:swap}, there exists a regular diagram $S'$ that is 
equivalent to $S$. The uniqueness follows from Lemma~\ref{lem:regular:unique}. 

%Finally,  Theorem~\ref{thm:equivalence} implies $\im(S)=\im(S')$.

It remains to show the second part of the theorem, that is, $S$ 
is regular if and only if the number of crossings in $S$ 
is less than or equal to that in each 
proper diagram $S''$ with $S''\sim S$. 

First, suppose that $S$ is a regular 
diagram and assume that $S'$ is a  proper diagram $S''$ with $S''\sim S$. 
By Proposition~\ref{prop:swap} and the first part of the theorem, 
diagram $S$ can be obtained from $S''$ by a sequence of strict swaps, 
and hence the number of crossings contained in $S$ is less than or equal to that contained in $S''$. 

Conversely, suppose that the number of crossings in $S$ is less 
than or equal to that in each proper diagram $S''$ with $S''\sim S$. 
If $S$ is not regular, then by the first part of the theorem 
there exists a regular diagram $S^*$ that is equivalent to $S$. 
By Proposition~\ref{prop:swap} and the first part of the 
theorem, diagram $S^*$ can be obtained from $S$ 
by a sequence of strict swaps, and hence 
the number of crossings contained in $S^*$ is less 
than that contained in $S$, a contradiction. Therefore $S$ is regular. 
\end{proof}

It is worth noting that combining 
Proposition~\ref{prop:swap} and Theorem~\ref{thm:regular}  
provides another characterisation of equivalence between proper diagrams, 
in addition to the one given in Theorem~\ref{thm:equivalence}. The proof 
of this fact is straightforward and hence omitted here. 

\begin{corollary}
	Suppose that $S$ and $S'$ are two proper diagrams.
	Then $S\sim S'$ if and only if there exists a sequence of 
	swaps $\epsilon_1,\dots,\epsilon_t$ for some $t\geq 0$ 
	such that $S'=\epsilon_t \epsilon_{t-1} \cdots \epsilon_1(S)$
	\eepf
\end{corollary}

\section{Poset Isomorphisms}
\label{sec:poset:iso}

In this section we study the two maps $\promap$ and $\bmap$ 
in Fig.~\ref{f:summary} introduced in Section~\ref{subsec:outline}. 
In particular, we will  prove the following result:

\begin{theorem}
	\label{thm:poset:binary}
	For $f\ge 3$, $r\ge 0$, and $k\ge 1$, the map 
	\begin{align*}
	\promap: ({\properd}^r_{f,k}, \preceq) \to (\mathcal M^r_{f+1,k}, \le) ~:~~~ S \mapsto \im(S)
	\end{align*}
	is a surjective poset homomorphism. Moreover, its restriction
	\begin{align*}
	\bmap: (\primed^r_{f,k}, \preceq) \to (\mathcal M^r_{f+1,k}, \le) ~:~~~ S \mapsto \im(S)
	\end{align*}
	is a poset isomorphism.
\end{theorem}

Using this theorem, we shall also show that in the case $r=0$ 
the composition $\tau^{-1}\bmap$ in Fig.~\ref{f:summary} gives a poset isomorphism
between $(\primed^0_{f,k}, \preceq)$ and $(\mathcal S_{f+1,k}, \subseteq)$
(see Theorem~\ref{isomorphism} below).

We  begin by introducing two types of operations on diagrams 
which are motived by the construction in~\cite[Theorem 2]{PW93}. 
Given a diagram $S$ with length $n \ge 2$,  
the first operation creates its {\em dual} diagram $S^*$ 
as follows. Denote the set of free sites in $S$ by $F(S)$, and 
create a new set of sites $H(S)$ consisting of sites of the form $i + \frac{1}{2}$ for all $1 \le i < n$.
Then remove each site in $F(S)$, and label all newly created 
sites (i.e. those in $H(S)$)  and relabel all other sites 
if necessary to obtain a diagram with length ${2n-1-|F(S)|}$ and $n-1$ free sites.  
To illustrate this process, in Fig.~\ref{f:dual} we depict the dual 
diagram $S^*$ of the diagram $S$ in Fig.~\ref{f:diagram}.  

\begin{figure}[ht]
	\begin{center}
		\includegraphics[width=0.9\textwidth]{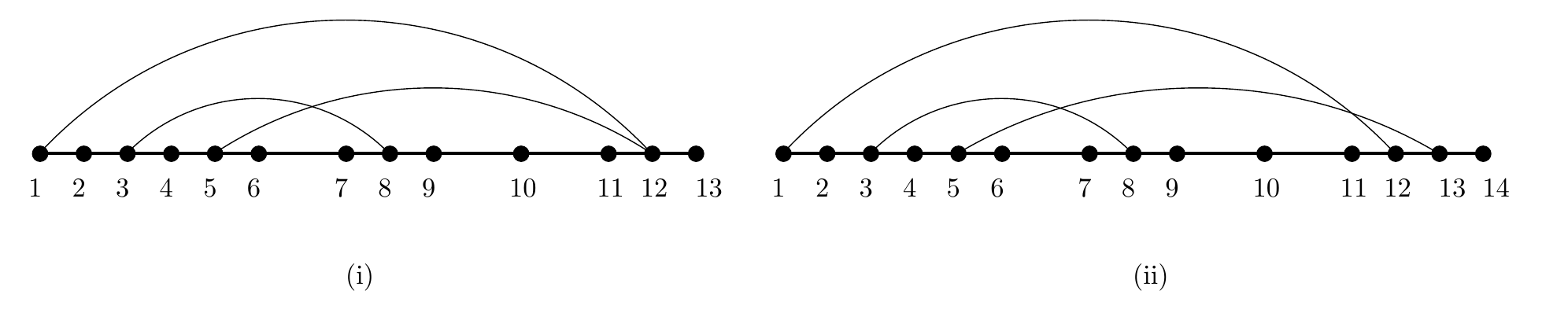}
	\end{center}
	\caption{Examples of dual and blow-up operations: (i) The dual of the diagram  $S$  in Fig.~\ref{f:diagram}.
		(ii) The blow-up of the diagram in (i).}
	\label{f:dual}
\end{figure}

Since each arc $(i,j)$ in $S$ satisfies $1<|i-j|<n-1$, the dual diagram $S^*$ 
contains no tiny arcs since each arc in $S^*$ covers at least two free sites. Moreover, 
no arc in $S^*$ covers all of the free sites in $S^*$. In addition, 
note that both $S$ and $S^*$ have the same size, that is, the same number of arcs.
The next technical lemma provides further connection between a diagram and its dual. 

\begin{lemma}
	\label{lem:dual:matrix}
	Given a diagram $S$ with length $n\ge 2$,  we have $\am(S)=\im(S^*)$.
\end{lemma}
\begin{proof}
Without loss of generality, suppose that the size of $S$ is $m$ for some $m\ge 1$ (as otherwise the lemma clearly follows).   
Let $(a_{i,j})=\am(S)$ and $(b_{i,j})=\im(S^*)$.  By construction, 
diagram $S^*$ has $f^*=n-1\ge 1$ free sites, and hence $(b_{i,j})$ has order $n$, the 
same as that of $(a_{i,j})$. Moreover, since the number of 
arcs in $S$ is $m$ and $\im(S^*)$ is symmetric,  %by Eq.~(\ref{eq:size}) 
it follows that the sum of the elements $b_{i,j}$ in $\im(S^*)$ is $2m$, which is the same as 
the sum of the elements $a_{i,j}$ in $\am(S)$. 

Because $\am(S)$ and $\im(S^*)$ are two symmetric non-negative integral 
matrices of order $n$ whose sum of entries is the same,
it suffices to show that for each $1\le i \le j \le n$, we 
have $b_{i,j} \ge a_{i,j}$. This is clearly the case if $a_{i,j}=0$ and 
hence we may assume that $i<j$ and $a_{i,j}=1$, that is, $(i,j)$ is an arc in $S$. 
Now let $(s,s')$ be the arc in $S^*$ that is derived from $(i,j)$ in $S$. 
%Denote the length of $S^*$ by $n^*$. 
Let $u^*_t~(1\le t \le f^*)$ be the $t$-th free site in $S^*$ and put $u_0^*=0$.
% and $u^*_{f^*+1}=n^*+1$. 
By construction, site $s$ is between $u^*_{i-1}$ and $u^*_i$, and hence 
is contained in block $B_i(S^*)$.  Similarly, we know that site $s'$ is contained in 
block $B_j(S^*)$. Thus $(s,s')$ is an arc between $B_i(S^*)$ and $B_j(S^*)$, from 
which we have $b_{i,j}\ge 1$, as required. 
\end{proof}

The second operation converts a non-binary diagram $S$ into a binary one $\blowup (S)$ 
as follows. 
Consider a site $s$ in a diagram $S$ supporting $b \ge 2$ arcs, that is, $s$ 
is base-paired with $b$ sites $s_1<s_2< \dots < s_b$. The {\em blow up} 
at $s$ results in the diagram $\blowup_s(S)$ that is obtained from 
$S$ by replacing the site $s$ with $b$ new sites $s'_1,s'_2,\dots,s'_b$, and 
for $1\le i \le b$, replacing the arc supported by $s$ and $s_i$ with a new arc supported by $s'_i$ and $s_i$. 
Finally all newly created sites are relabelled by consecutive integers and all other sites are relabelled
if necessary.  Note that each of the newly created sites supports 
precisely one arc, and hence the number of sites in $\blowup_s(S)$ that support
at least two arcs is one less than that in $S$. We now continue this process until a binary diagram, called the 
{\em blow up} of $S$ and denoted by $\blowup(S)$,  is obtained. 
See Fig.~\ref{f:dual} for an example. Note that $\im(S)=\im(\blowup(S))$.

Using the dual and blow up operations, we 
now present a characterisation of matrices that can be realized as the block matrix of a 
proper diagram, which will be key in proving Theorem~\ref{thm:poset:binary}.

\begin{theorem}
	\label{thm:surjective:matrix}
	Suppose that $\mathbf M$ is a symmetric non-negative non-trivial integral matrix. Then 
	there exists a  proper diagram $S$ with $\im(S)=\mathbf M$ if and only if all 
	diagonal and rainbow elements in $\mathbf M$ are zero.
	Moreover, if $\mathbf M$ is a $(0,1)$-matrix, then $S$ contains no parallel arcs. 
\end{theorem}
\begin{proof}
The ``only if" direction follows from 
Lemma~\ref{lem:incident:matrix}. To establish the ``if" direction, we shall first prove the following \\

\noindent
{\bf Claim:}  Given a symmetric non-trivial $(0,1)$-matrix $(b_{i,j})=\mathbf B$ of order $m\ge 1$ whose diagonal and rainbow elements are zero, there exists a proper diagram $S$ containing no parallel arcs such that $\im(S)=\mathbf B$. \\

\noindent
{\bf Proof of the Claim:} Since $\mathbf B$ is non-trivial, it contains at least one non-zero element and hence by the assumptions of the claim we have $m\ge 3$.

Let $H=\{1\le i < m : b_{i,i+1} \ge 1 \}$ be the index set of the non-zero 
superdiagonal elements in $\mathbf B$. Consider the matrix $\mathbf B'$ 
obtained from $\mathbf B$ by replacing both $b_{i,i+1}$ and $b_{i+1,i}$ with 
zero for each index $i$ in $H$. Then all diagonal, semi-diagonal and rainbow 
elements in $\mathbf B'$ are zero. By Proposition~\ref{prop:iso:rna}, there exists a 
diagram $S'$ with $\am(S')=\mathbf B'$. Now let $S^*$ be the 
dual of $S'$, and let $\blowup$ be the blow up of $S^*$ so that 
$S''=\blowup(S^*)$ is a binary diagram. Then we have 
$$\im(S'')=\im\left(\blowup(S^*)\right)=\im(S^*)=\am(S')=\mathbf B',
$$
where the third equality follows from Lemma~\ref{lem:dual:matrix}. Using 
Lemma~\ref{lem:incident:matrix} it follows that $S''$ is a proper diagram containing neither tiny nor parallel arcs.

Now for each index $i$ in $H$, insert a new site $a_i$ in 
block $B_i(S'')$ and a new site $a_j$ in block $B_{i+1}(S'')$, and add an arc 
between these two newly added sites. 
Finally, relabel all the sites to obtain a diagram.

Let $S$ be the diagram resulting from $S''$ after performing 
the above operation for each of the indices in $H$. 
Then, by construction, $S$ is a proper diagram which has the same 
number of free sites as that of $S''$, such that  $S$ does not contain any parallel arcs. Moreover, we have $\im(S)=\mathbf B$, from which the claim follows.  
\eepf

Now we proceed to establish the ``if" direction. To this end, assume that $(x_{i,j})=\mathbf M$ is a symmetric 
non-negative non-trivial integral matrix of order $m$ whose diagonal and 
rainbow elements are zero. Let ${\mathbf M}'=(x'_{i,j})$ be the maximal symmetric 
$(0,1)$-matrix of order $m$ dominated by $\mathbf M$, that is,  where $x'_{i,j}=1$ if and only if $x_{i,j} \ge 1$ holds. Then the diagonal and rainbow elements in $\mathbf M'$ are 0. 
By the {\bf Claim}, there exists a diagram $S'$ with $\im(S')={\mathbf M}'$. Let 
$$
A=\{(i,j)\,:\, 1\le i < j \le m~~\mbox{and}~~x_{i,j}-x'_{i,j}\ge 1\}
$$
be the set of positions above the diagonal in which the element in $\mathbf M$ differs from that in $\mathbf M'$. 

For each position $(i,j)$ in $A$, put $t=x_{i,j}-x'_{i,j}$. Then we 
have $t\geq 1$. Now insert $t$ new sites $a_1,\dots,a_t$ in block $B_i$ 
and $t$ new sites $b_1,\dots,b_t$ in block  $B_j$, relabel 
all of the sites, and for $1\le p \le t$, add an arc between 
the two newly added sites $a_p$ and $b_p$.  
Denote the resulting binary diagram obtained from $S'$ after completing 
the above operation for each position in $A$ by $S$. Then by construction it follows that $S$ is binary diagram with $\im(S)= \mathbf M$. By Lemma~\ref{lem:incident:matrix} we can further conclude that $S$ is a proper diagram, from which the theorem follows.  
\end{proof}

Using this last result, we now prove Theorem~\ref{thm:poset:binary}.

\bigskip
\noindent
{\bf Proof of Theorem~\ref{thm:poset:binary}:}  Given a diagram $S$ in $\properd^r_{f,k}$, we first 
show that $\im(S)$ is contained in   ${\mathcal M}^r_{f+1,k}$. 
Indeed, because $S$ contains $f$ free sites, the order of $\im(S)$ is $f+1$. As $S$ is a proper diagram with at least one arc, $\im(S)$ is a non-trivial symmetric non-negative integer matrix, and by Lemma~\ref{lem:incident:matrix} all diagonal and rainbow elements in $\im(S)$ are zero. 
Since $S$ is $k$-noncrossing, by Lemma~\ref{lem:sturcutre:matrix:noncrossing} 
it follows that $\im(S)$ is also $k$-noncrossing. 
Moreover, by $\mdr(\im(S))=\mdr(S)\le r$ we have $\im(S)\in {\mathcal M}^r_{f+1,k}$, 
as claimed. Therefore, $\promap$ is indeed a poset homomorphism from  
$(\properd^r_{f,k} , \preceq) $ to $({\mathcal M}^r_{f+1,k}, \le)$.

The next step is to show that the map $\promap$ is surjective. To this end,
fix a matrix $\mathbf M$ in ${\mathcal M}^r_{f+1,k}$. 
Since $\mathbf M$ is a symmetric non-negative integer matrix 
whose diagonal and rainbow entries are zero, 
by Theorem~\ref{thm:surjective:matrix} there exists a 
proper  diagram $S$ with $\im(S)=\mathbf M$. 
Note that $S$ contains precisely $f$ free sites.  
By Theorem~\ref{thm:regular}, we may assume that $S$ 
is regular as otherwise we can replace it with the 
regular diagram in its equivalence class. 
Since $\im(S)$ is $k$-noncrossing and $S$ is regular, 
by Proposition~\ref{prop:noncrossing:matrix} it follows that 
$S$ is also $k$-noncrossing. Together 
with $\mdr(S)=\mdr(\im(S))\le r$, we conclude that $S$ 
is contained in $\properd^r_{f,k}$, and hence the map $\promap$ is surjective, as required. 

The last step is to  show that map $\bmap$ is a poset isomorphism.
% between $(\primed^r_{f,k} , \preceq)$ and $({\mathcal M}^r_{f+1,k}, \le)$. 
Since $\promap$ is surjective, by Theorems~\ref{thm:equivalence} and~\ref{thm:regular} it 
follows that map  $\bmap$ is a surjective poset 
homomorphism between $(\primed_{f,k}^r, \preceq)$ and 
$(\mathcal M^r_{f+1,k}, \le)$. Moreover, this map is 
injective by Theorems~\ref{thm:equivalence} and~\ref{thm:regular}.  
\eepf

Rephrased in our terminology, Penner and Waterman established a canonical 
isomorphism between the poset $(\primed^0_{f,1},\preceq)$, which 
consists of binary $1$-noncrossing diagrams with $f$ free sites 
and containing neither parallel nor tiny arcs, and 
the poset $(\mathcal S_{f+1,1},\le)$, which consists 
of all $1$-noncrossing diagrams with at least one arc and length $f+1$ ~\cite[Theorem 2]{PW93}. 
We now conclude this section by using Theorem~\ref{thm:poset:binary} to 
generalize this result to noncrossing diagrams (cf. Fig.~\ref{f:summary}).

\begin{theorem}
	\label{isomorphism}
	For $f\ge 3$ and $k\ge 1$, the map 
	\begin{align*}
	\tau^{-1}\promap: ({\properd}^0_{f,k}, \preceq) \to (\mathcal S_{f+1,k}, \subseteq)~:~
	S \mapsto \am^{-1}(\im(S))
	\end{align*} 
	is a surjective poset homomorphism. Moreover, its restriction to regular diagrams
	$$
	\tau^{-1}\bmap: (\primed^0_{f,k}, \preceq) \to (\mathcal S_{f+1,k}, \subseteq)
	~:~
	S \mapsto \am^{-1}(\im(S))
	$$
	is a poset isomorphism. 
\end{theorem}

\begin{proof}
This clearly follows from Proposition~\ref{prop:iso:rna}, Theorem~\ref{thm:poset:binary}, and the fact that 
$\mathcal M_{f+1,k}=\mathcal M^0_{f+1,k}$.
\end{proof}

Note that since  $\properd^0_{f,1}=\primed^0_{f,1}$, the fact that $\properd^0_{f,1}$ and $\mathcal S_{f+1,1}$ are isomorphic~\cite[Theorem 2]{PW93}  is also a consequence of the last result.

\section{Spaces of $k$-noncrossing Diagrams}
\label{sec:space}

In this section we prove our main results.
We first investigate some properties of the poset 
${(\mathcal S_{m,k}, \subseteq)}$, $m\ge 2k+1$, as defined in Section~\ref{subsec:diagram}.
To do this we shall use a generalisation of triangulations mentioned in the introduction
called $k$-triangulations (also known as multitriangulations)~\cite{dress2002line,nakamigawa2000generalization,stump2011new}, 
whose definition we now recall.

\begin{figure}[ht]
	\begin{center}
		\includegraphics[width=0.9\textwidth]{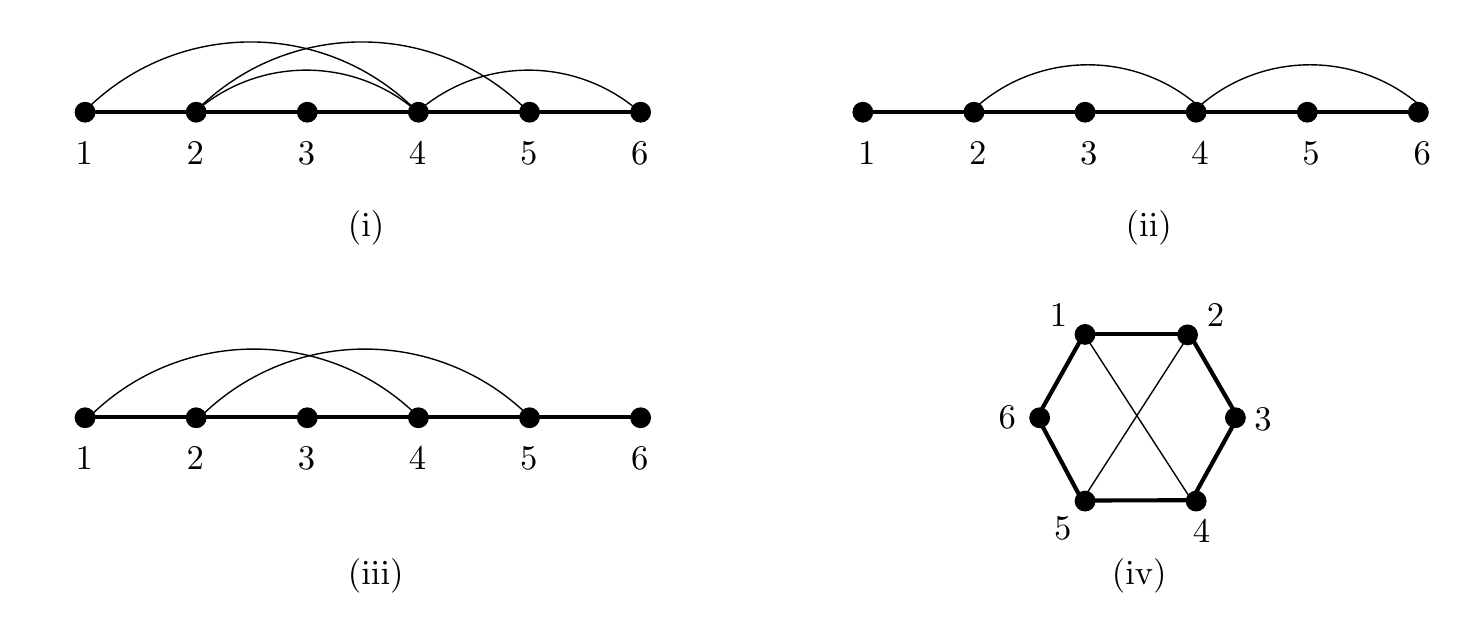}
	\end{center}
	\caption{Examples of diagrams and $k$-triangulations: 
		(i)A diagram $S$ in $\mathcal {S}_{6,2}$;
		(ii) A diagram $S_1$ in $\mathcal {S}^*_{6,2}$;
		(iii) A diagram $S_2$ in $\mathcal {S}^o_{6,2}$.   
		Note that both $S_1$ and $S_2$ are also diagrams 
		in $\mathcal {S}_{6,2}$. Moreover, $\kappa(S)=(S_1,S_2)$ 
		holds for the map $\kappa$ defined in the proof of Lemma~\ref{lem:poset:decomposition}. 
		(iv) A face $T$ in $\mathcal T_{6,2}$.  It consists of  
		diagonals $1-4$ and $2-5$, both of which are $2$-relevant. 
		Note that $\theta(T)=S_2$ holds for the map  $\theta$  in  Lemma~\ref{lem:top:sphere}.  
	}
	\label{f:diagram:triangulation}
\end{figure}

Consider the convex $m$-gon with vertices $\{1,2,\dots,m\}$. 
A diagonal between two vertices $i$ and $j$ in the $m$-gon, denoted 
by $i-j$, is called {\em $k$-relevant} if the length of the shortest 
path between these two vertices in the $m$-gon is greater than $k$, that is, $k < |i-j| < m-k$ holds
(see Fig.~\ref{f:diagram:triangulation}(iv) for an example). 
Clearly, each diagonal in a set of $(k+1)$ pairwise crossing diagonals 
must be $k$-relevant. Denote the set of $k$-relevant diagonals 
in the $m$-gon by $\Gamma_{m,k}$, so that each set of $(k+1)$ pairwise 
crossing diagonals is a subset of $\Gamma_{m,k}$, and 
let $\mathcal T_{m,k}$ be the simplicial complex with vertex 
set $\Gamma_{m,k}$ whose faces are the nonempty subsets of $\Gamma_{m,k}$ 
which do not contain $(k+1)$ pairwise crossing  diagonals. 
A maximal face in $\mathcal T_{m,k}$ is called a {\em $k$-triangulation of the 
	$m$-gon}~\cite{stump2011new} . 
For example, 
the face depicted in  Fig.~\ref{f:diagram:triangulation}(iv) is a $2$-triangulation of the $6$-gon. 
Note that some authors use the term $k$-triangulation in a slightly 
different way, allowing diagonals that are not necessarily $k$-relevant (see, e.g.~\cite{pilaud2009multitriangulations}).

Motivated by these definitions, we say an arc $e=(s,s')$ in a 
diagram with length $m$ is $k$-relevant if $k < |s-s'| < m-k$.
Let $\mathcal {S}^o_{m,k}$ be the set containing all diagrams $S$ in $\mathcal {S}_{m,k}$ 
such that every diagonal in $S$ is $k$-relevant, and 
$\mathcal {S}^*_{m,k}$ denote the set consisting of all diagrams $S$ in $\mathcal {S}_{m,k}$ 
in which no diagonal in $S$ is $k$-relevant. 
See Fig~\ref{f:diagram:triangulation}(ii-iii) for 
examples of diagrams in $\mathcal {S}^*_{6,2}$ and $\mathcal {S}^o_{6,2}$. 
Note that $\mathcal {S}^*_{m,1}=\emptyset$
and $\mathcal {S}^o_{m,1}=\mathcal {S}_{m,1}$.
Moreover, $\mathcal {S}^*_{m,k}$ and $\mathcal {S}^o_{m,k}$ are 
two disjoint subposets of $\mathcal {S}_{m,k}$ under $\subseteq$. 

The arc set of any diagram $S$ in $\mathcal S_{m,k}$ 	
can be partitioned into two subsets: 
the one that contains all $k$-relevant arcs in $S$ 
and the other its complement. In order to 
deal with the case that one of these two subsets is empty,  
we extend the posets $\mathcal {S}^*_{m,k}$ and $\mathcal {S}^o_{m,k}$ 
by adding a new bottom element to each of them, which
we denote by $0^*$ and $0^o$, respectively. 
We denote these new posets by $\mathcal {\hat{S}}^*_{m,k}$ and 
$\mathcal {\hat{S}}^o_{m,k}$, respectively.
Note that we are assuming that $0^o$ 
is distinct from $0^*$, and hence $\mathcal {\hat{S}}^*_{m,k}$ 
and $\mathcal {\hat{S}}^o_{m,k}$ are disjoint. 

The next lemma will allow us to reduce the problem of
understanding the topology of $|\Delta(\mathcal {S}_{m,k})|$ to 
that of understanding the topology of $|\Delta(\mathcal {S}^*_{m,k})|$
and $|\Delta(\mathcal {S}^o_{m,k})|$.

\begin{lemma}
	\label{lem:poset:decomposition}
	For $k\ge 1$ and $m \ge 2k+1$, the poset $\mathcal {S}_{m,k}$ is isomorphic 
	to the open interval $(\mathcal {\hat{S}}^*_{m,k} \times \mathcal {\hat{S}}^o_{m,k})_{>(0^*,0^o)}$. 
	Moreover, we have
	\begin{equation}
	|\Delta(\mathcal {S}_{m,k})| \cong |\Delta(\mathcal {S}^*_{m,k})|*|\Delta(\mathcal {S}^o_{m,k})| .
	\end{equation}
\end{lemma}

\begin{proof}
First, consider the map $\kappa_*$ that 
takes a diagram $S$ in $\mathcal {S}_{m,k}$
to the diagram $\kappa_*(S)$ in $\mathcal {\hat{S}}^*_{m,k}$ 
that consists of all arcs $e$ in $S$ such that $e$ is not $k$-relevant.
In case each of the arcs in $S$ is $k$-relevant, 
we set $\kappa_*(S)=0^*$. 
Similarly, let $\kappa_o$ be the map that takes 
a diagram $S$ in $\mathcal {S}_{m,k}$ to the 
diagram $\kappa_o(S)$ in $\mathcal {\hat{S}}^o_{m,k}$ 
that consists of all $k$-relevant arcs in $S$ or to $0^o$ 
if none of the arcs in $S$ is $k$-relevant. 
Then it is straightforward to check that $\kappa_*$
and $\kappa_o$ are both poset homomorphisms. 

Now consider the map $\kappa$ that takes a diagram $S$ in $\mathcal {S}_{m,k}$ to $(\kappa_*(S),\kappa_o(S))$. 
See Fig.~\ref{f:diagram:triangulation}(i-iii) for an example in which $\kappa(S)=(S_1,S_2)$.  
Then $\kappa$ is a poset homomorphism from $\mathcal {S}_{m,k}$ to 
$\mathcal {\hat{S}}^*_{m,k} \times \mathcal {\hat{S}}^o_{m,k}$. 
Moreover, since $S$ contains at least one arc, it follows that $(\kappa_*(S),\kappa_o(S))\not = (0^*,0^o)$.

On the other hand, given a pair $(S_1,S_2) \in \mathcal {\hat{S}}^*_{m,k} \times \mathcal {\hat{S}}^o_{m,k}$ 
that is distinct from $(0^*,0^o)$, we consider the necessarily non-trivial
diagram $S$ that is obtained by taking the union of the arcs in $S_1$ and $S_2$. 
Then $S$ is a diagram in $\mathcal {S}_{m,k}$ since
$S_1$ contains no $k$-relevant arc, and hence no arc in $S_1$ can be contained in a set of $(k+1)$ pairwise crossing arcs. Furthermore, 
by construction we have $\kappa(S)=(S_1,S_2)$. 
Therefore  $\kappa$ is a surjective poset homomorphism from 
$\mathcal {S}_{m,k}$ to $(\mathcal {\hat{S}}^*_{m,k} \times \mathcal {\hat{S}}^o_{m,k})_{>(0^*,0^o)}$.
Note that given two distinct diagrams $S$ and $S'$ in  $\mathcal {S}_{m,k}$, 
there exists at least one arc that is contained in one but not both diagrams, and 
hence $\kappa(S)\not =\kappa(S')$ follows. This implies that $\kappa$ is also injective, 
and thus $\kappa$ is an isomorphism.

Finally, we have 
%that $\mathcal {\hat{S}}^o_{m,k})_{>0^o}$
\begin{eqnarray*}
	|\Delta(\mathcal {S}_{m,k})| &\cong& |\Delta\big((\mathcal {\hat{S}}^*_{m,k} \times \mathcal {\hat{S}}^o_{m,k})_{>(0^*,0^o)}\big)|\cong |\Delta\big((\mathcal {\hat{S}}^*_{m,k})_{\supset\, 0^*}\big)| * |\Delta\big((\mathcal {\hat{S}}^o_{m,k})_{\supset\, 0^o}\big)| \\
	&\cong& |\Delta(\mathcal {S}^*_{m,k})|*|\Delta(\mathcal {S}^o_{m,k})| .
\end{eqnarray*}
Here the second $\cong$ follows from Eq.~(\ref{eq:join:product}), and the third $\cong$ 
follows from the fact that $(\mathcal {\hat{S}}^*_{m,k})_{\supset\, 0^*}$ and 
$(\mathcal {\hat{S}}^o_{m,k})_{\supset \, 0^o}$ are isomorphic to $\mathcal {{S}}^*_{m,k}$ and  $\mathcal {{S}}^o_{m,k}$, respectively. 
\end{proof}

In view of the last lemma,  the topology of  $|\Delta(\mathcal {S}_{m,k})|$ is determined 
by the topology of $|\Delta (\mathcal {S}^*_{m,k})|$ and $|\Delta (\mathcal {S}^o_{m,k})|$. 
It is straightforward to see that $|\Delta (\mathcal {S}^*_{m,k})|$ is a 
simplex of dimension $m(k-1)-1$. To obtain the topology of $|\Delta (\mathcal {S}^o_{m,k})|$, 
we define a natural map $\theta$ which takes any face $T$ in 
the face poset $\fp(\mathcal T_{m,k})$ of $\mathcal T_{m,k}$
with diagonal set $\{i_1-j_1, \cdots, i_t-j_t\}$ 
to the diagram ${\theta}(T)$ in $\mathcal{S}^o_{m,k}$ with arc 
set $\{(i_1,j_1),\cdots,(i_t,j_t)\}$. For example, $\theta$ maps the face in Fig.~4(v) to the diagram in Fig.~4(iii). 
We now show that this yields a poset isomorphism.

\begin{lemma}
	\label{lem:top:sphere}
	For two integers $k \ge 1$ and $m \ge 4$ with $m \ge 2k+1$, the map $\theta$ 
	induces a poset isomorphism from the face poset $\fp(\mathcal T_{m,k})$ 
	to $\mathcal {S}^o_{m,k}$. Moreover, the 
	simplicial complex $\Delta (\mathcal {S}^o_{m,k})$ is a 
	simplicial sphere of  dimension $k(m-2k-1)-1$.
\end{lemma}

\begin{proof}
It is straightforward to verify that $\theta$ is indeed a poset isomorphism.
%from $\fp(\mathcal T_{m,k})=(\mathcal T_{m,k},\subseteq)$ to $(\mathcal {S}^o_{m,k},\subseteq)$. 
Therefore  $\Delta (\mathcal {S}^o_{m,k})$ is isomorphic to 
$\Delta(\fp(\mathcal T_{m,k}))$. Note 
that $\Delta(\fp(\mathcal T_{m,k}))$ is the (first) 
barycentric subdivision of $\mathcal T_{m,k}$ and 
hence it is isomorphic to $\mathcal T_{m,k}$ (see, e.g.~\cite[p. 1844]{bjorner1995topological}). 
The lemma now follows from the fact that $\mathcal T_{m,k}$ 
is a simplicial sphere with dimension $k(m-2k-1)-1$ 
(see, e.g.  Statement 1.3 and Theorem 1.1 in~\cite{stump2011new}). 
\end{proof}

Since $\mathcal {S}_{m,1}=\mathcal {S}^o_{m,1}$ and $\mathcal T_{m,1}$ 
is the arc poset $\mathcal T_{m}$ of  triangulations of an $m$-gon 
mentioned in Section~\ref{sec:introduction}, note that Lemmas~\ref{lem:poset:decomposition} 
and~\ref{lem:top:sphere} can be regarded as a generalization of the 
result relating $1$-noncrossing diagrams to triangulations of a polygon 
as stated in \cite[Theorem 2 and Proposition 3]{PW93}.

We now determine the topology of $|\Delta({\mathcal S}_{m,k})|$.  Note that 
the fact  that $|\Delta({\mathcal S}_{m,1})|$ 
is a topological sphere of dimension $m-4$ as established in~\cite{PW93} 
is a special case of the following result with $k=1$.

\begin{proposition}
	\label{prop:join}
	For any two integers $k\ge1$ and $m \ge 4$ with $m \ge 2k+1$, the 
	poset $\mathcal {S}_{m,k}$ is pure and of rank $2k(m-k)-k-m$. 
	Moreover, $|\Delta (\mathcal {S}_{m,k})|$ is homeomorphic to the join of a 
	simplicial sphere of dimension $k(m-2k-1)-1$ and an $\left(m(k-1)-1\right)$-simplex.
\end{proposition}

\begin{proof}
Since $\mathcal {S}^o_{m,k}$ and $\mathcal {S}^*_{m,k}$ are 
both pure, by Lemma~\ref{lem:poset:decomposition} it 
follows that the poset $\mathcal {S}_{m,k}$ is pure and of 
rank $2k(m-k)-k-m$.  The proposition now follows 
from Lemmas~\ref{lem:poset:decomposition} and~\ref{lem:top:sphere}, 
and the fact that $\Delta (\mathcal {S}^*_{m,k})$ is an $(m(k-1)-1)$-simplex. 
\end{proof}

Using Theorem~\ref{isomorphism} and Proposition~\ref{prop:join}, we 
immediately obtain our two main results \rev{as stated in Section~\ref{sec:introduction}}, the first of which follows immediately:

%\begin{theorem}\label{mainresult1}
	\begin{customthm}{1.1}
	%[Theorem~\ref{intro:mainresult1}]
	For two integers $k\ge1$ and $f \ge 3$ with $f \ge 2k$, 
	$|\Delta(\primed^0_{f,k})|$ is the join of 
	a simplicial sphere of dimension $k(f-2k)-1$ and an $\left((f+1)(k-1)-1\right)$-simplex.
	\eepf
\end{customthm}	
%\end{theorem}

%\begin{theorem}\label{mainresult2} 		[Theorem~\ref{intro:mainresult2}]
	%\label{thm:main}
		\begin{customthm}{1.2}
	For any $r\ge 0$, $f \ge 3$ and $k\ge 1$ with $f\ge 2k$,  the poset $(\primed_{f,k}^r, \preceq)$ 
	is pure and of rank $k(2f-2k+1)+r-f-1$. 
\end{customthm}	
%\end{theorem}
\begin{proof}
By Proposition~\ref{prop:join} and Theorem~\ref{isomorphism}, we know
that $(\primed^0_{f,k}, \preceq)$ is pure and of rank $2k(f-k)+k-f-1$. 
Now suppose that $S_1\prec \cdots \prec S_t$ is a maximal 
chain in $\primed_{f,k}^r$ for some $t>0$. 
Then we have $\mdr(S_t)=r$ since otherwise there 
exists a diagram $S^*_t$ in $\primed_{f,k}^r$ 
with $ S_t \prec S^*_t$, a contradiction. Next, 
for each $1< i \le t$, diagram $S_i$ contains 
precisely one more arc than $S_{i-1}$. Thus we have $t>r$. 
In addition, this maximal chain induces precisely 
one maximal chain $S'_1\prec \cdots \prec S'_{t-r}$ in $(\primed^0_{f,k}, \preceq)$ \
by removing all tiny arcs, and all but one arc in each set of parallel arcs.  
Therefore, the length of each maximal chain in $\primed_{f,k}^r$ 
is precisely $r$ plus the length of a maximal chain 
in $\primed^0_{f,k}$, from which the theorem immediately follows.
\end{proof}

\section{Discussion}
\label{sec:discussion}

Using the topology of $|\Delta(\primed^0_{f,1})|$ and 
the fact that the natural inclusion $\primed^r_{f,1}\subseteq \primed^{r+1}_{f,1}$ 
is a chain homotopy, in \cite[Theorem 7]{PW93} Penner and 
Waterman show that $|\Delta(\primed^r_{f,1})|$ has the homology of a $(f-3)$-dimensional sphere 
for $f\ge 3$ and $r\ge 0$.
It would therefore be of interest to see whether or not the following 
holds: For any $r\ge 0$, $f \ge 3$ and $k\ge 1$,  $|\Delta(\primed_{f,k}^r)|$ has the 
same homology as the join of a simplicial sphere 
of dimension $k(f-2k)-1$ and an $((f+1)(k-1)-1)$-simplex. One approach
to try proving this would be to first understand whether or not
the natural inclusion $\primed^r_{f,k}\subseteq \primed^{r+1}_{f,k}$ is a chain homotopy.

Our results also lead to several interesting counting problems.
For example, can formulae be found for the number of
elements in $\primed_{f,k}^r$ or
$\properd_{f,k}^r$? In addition, using	
Lemma~\ref{lem:top:sphere} and~\cite[Theorem 1.2]{stump2011new} (see,also~\cite{jonsson2005generalized,krattenthaler2006growth}) 
we can give a  formula for the number of facets in $\Delta (\mathcal {S}^o_{m,k})$. 
It would be interesting to see whether Theorem~\ref{isomorphism} 
and Lemma~\ref{lem:poset:decomposition},  could be used together 
with this formula to determine the number of  facets in $\Delta(\primed^0_{f,k})$. 

\rev{Finally, it could be of interest to investigate how the spaces studied 
here might help to understand properties of structures 
of RNA molecules in more practical applications. 
For example, although it is known that finding minimal energy 
pseudoknots is NP-complete under typical RNA 
minimum free energy models (see e.g. \cite{akutsu2000dynamic}), 
there are efficient algorithms for restricted classes
of pseudoknots (see e.g. \cite{reidys2011topology}). 
Hence it could be worth investigating if there is any 
relationship between the Penner-Waterman poset and collections of
RNA structures satisfying such restrictions, e.g. genus restricted structures \cite{reidys2011topology} 
or grammar-based structures  (see e.g. \cite{condon2004classifying, nebel2012algebraic}).
Note that multitriangulations have connections with 
several combinatorial structures such as polyonimoes and twisted surfaces \cite{pilaud2009multitriangulations}, 
which indicates that there may be some interesting relationships to uncover.} %\cite{reidys2010combinatorial.}

\rev{In another direction it could be interesting to see if
our results might be used to help define 
and study combinatorial landscapes of $k$-noncrossing RNA 
structures.  Combinatorial landscape theory involves the study of 
configuration spaces (e.g. the space of RNA sequences with fixed length), 
that come equipped with some notion 
of adjacency (e.g. RNA sequences differing in one nucleotide are adjacent), 
on which some fitness map is defined that 
assigns a real number to each element in the space  
(e.g. the energy of a minimum free energy structure for an RNA sequence) \cite{reidys2002combinatorial}. 
This concept is useful in, for example, combinatorial optimization 
where it can be used to design algorithms for finding (locally) optimal elements
in the underlying space. %(see e.g. \cite[Chapter 4]{felsenstein2004inferring}).
There is a natural notion of adjacency for multitriangulations that is
given in terms of ``flipping" around an edge in a multitriangulation to give another multitriangulation
(this generalizes the well-known notion of 
flippings in triangulations -- see e.g. \cite{sleator1988rotation}). 
Moreover,
% any multitriangulation  {\bf in the set $\mathcal T_{m,k}$} can be obtained from any other in the same set through a sequence of flips (see e.g. \cite{pilaud2009multitriangulations} and the references therein).  
 any two multitriangulations in $\mathcal T_{m,k}$ can be converted from one to the other through a sequence of flips (see e.g. \cite{pilaud2009multitriangulations} and the references therein).  
Thus, it would be interesting to explore 
if these results coupled with our above observations concerning multitriangulations 
might be used as a starting point for investigating 
landscapes of $k$-noncrossing RNA structures.}

\bibliographystyle{siamplain}
%\bibliography{references}
\bibliography{2_RNA.bib}
\end{document}